\documentclass[a4paper, 11pt]{article}
\usepackage{anysize}
\marginsize{3,5cm}{2,5cm}{2,5cm}{2,5cm}%tutaj ustawiasz marginesy
\usepackage{amssymb}
\usepackage{amsmath,amsbsy,amssymb,amscd}
\usepackage{t1enc}
\usepackage[cp1250]{inputenc}
\usepackage[all]{xy}
\usepackage{color}
\usepackage{amsfonts}
\usepackage{latexsym}
\usepackage{amsthm}
\usepackage{mathrsfs}
\usepackage{hyperref}
\usepackage{graphicx}
\usepackage{comment}

\usepackage{color}
\usepackage{caption}
\usepackage{subcaption}
\usepackage{enumerate}
\usepackage{xcolor}
\usepackage[draft]{todonotes}
\usepackage{float}
\usetikzlibrary{arrows, calc, decorations.markings, positioning, quotes}

\definecolor{orange}{rgb}{1,0.5,0}

\usepackage{mathrsfs} %Script dla zapisu \mathscr{}

\DeclareMathAlphabet{\mathpzc}{OT1}{pzc}{L}{it} %stylizowany
\theoremstyle{definition}
\newtheorem{definition}{Definition}[section]
\newtheorem{theorem}[definition]{Theorem}

\newtheorem{proposition}[definition]{Proposition}
\newtheorem{corollary}[definition]{Corollary}
\newtheorem{lemma}[definition]{Lemma}

\newtheorem{remark}[definition]{Remark}

\def\Im{\mathrm{Im\,}}

\def\geq{\geqslant}
\def\leq{\leqslant}
\def\R{\mathbb{R}}

\def\Z{\mathbb{Z}}
\def\N{\mathbb{N}}

\def\id{\mathrm{id}}

\def\epsilon{\varepsilon}

\def\mf{\mathfrak}
\def\Lie{\operatorname{Lie}}
\def\ad{\operatorname{ad}}
\def\Ad{\operatorname{Ad}}

\def\diag{\operatorname{diag}}

\def\inj{\operatorname{inj}}

\def\Bow{\operatorname{Bow}}

\def\ker{\operatorname{ker}}
\def\vspan{\operatorname{span}}
\def\Sym{\operatorname{Sym}}
\def\Leb{\operatorname{Leb}}

\newcommand{\bea}{\begin{eqnarray}}
  \newcommand{\eea}{\end{eqnarray}}
  \newcommand{\beab}{\begin{eqnarray*}}
  \newcommand{\eeab}{\end{eqnarray*}}

  \newcommand{\be}{\begin{equation}}
  \newcommand{\ee}{\end{equation}}

\newcommand{\set}[1]{\left\lbrace #1 \right\rbrace}
\newcommand{\abs}[1]{\left| #1 \right|}
\newcommand{\mc}{\mathcal}
\newcommand{\norm}[1]{\abs{\abs{#1}}}

\newcommand{\of}{\circ}

\newcommand{\mbf}{\mathbf}
\newcommand{\ve}{\epsilon}

\title{Slow entropy of higher rank abelian unipotent actions}
\author{Adam Kanigowski\footnote{Department of Mathematics, University of Maryland, College Park, MD 20742, USA, E-mail: adkanigowski@gmail.com A. K. was partially supported by the NSF grant DMS-1956310}, Philipp Kunde\footnote{Department of Mathematics, The Pennsylvania State University, University Park, PA 16802, USA, E-mail: pkunde.math@gmail.com. P.K. acknowledges financial support from a DFG Forschungsstipendium under Grant No. 405305501.}, Kurt Vinhage\footnote{Department of Mathematics, The Pennsylvania State University, University Park, PA 16802, USA, E-mail: kwv104@psu.edu}, Daren Wei\footnote{Department of Mathematics, The Pennsylvania State University, University Park, PA 16802, USA, E-mail: duw170@psu.edu  D. W. was partially supported by the NSF grant DMS-16-02409}}

\begin{document}
\maketitle
\begin{abstract}
We study slow entropy invariants for abelian unipotent actions $U$ on any finite volume homogeneous space $G/\Gamma$. For every such action we show that the topological slow entropy can be computed directly from the dimension of a special decomposition of $\Lie(G)$ induced by $\Lie(U)$. Moreover, we are able to show that the metric slow entropy of the action coincides with its topological slow entropy. As a corollary, we obtain that the complexity of any abelian horocyclic action is only related to the dimension of $G$. This generalizes the rank one results from \cite{KanigowskiVinhageWei} to higher rank abelian actions.
\end{abstract}
\section{Introduction}

Since its introduction, metric and topological entropies have played a central role in structural questions for dynamical systems and have been a crucial tool in detecting chaoticity of a system. Hallmark results include the Ornstein isomorphism theorem, Sinai's theorem providing a maximal entropy Bernoulli factor, and its surprising ability to detect smooth structures as shown in the Katok entropy conjecture. The relation between the metric and topological entropy is encaptured in the {\em variational principle}. The subject of this paper is a confluence of two of its generalizations: entropy for abelian group actions and slow entropy.

While the standard definition of entropy has a straightforward generalization to group actions of amenable groups, and this definition has several useful applications for symbolic systems, it has a critical failure when applying it to group actions by diffeomorphisms: it is always zero. Several adaptations of the definition are available which are functional for {\it hyperbolic} actions of abelian groups, such as the Fried average entropy studied in \cite{KatokKatokRH}.

Another such adaptation (also studied in \cite{KatokKatokRH} or \cite{Hochman}) is the {\it slow entropy} of a group action, which, rather than counting the {\it exponential} growth rate of distinguishable orbit types, counts the growth rate at a scale that can be chosen for the considered dynamics. Slow entropy was first introduced in \cite{Kat-Tho} as an isomorphism invariant for actions of amenable groups, with introducing an invariant for smooth, hyperbolic abelian group actions as one of its main goals.

However, it was discovered to have another application, that of detecting subexponential growth rates for flows and transformations (actions of $\R$ and $\Z$). When working with slow entropy it is first important to determine the scale which describes the growth rates and also distinguishes different systems. For the usual definition of entropy, the number of orbits grows at exponential rate $e^{ht}$ with time $t$. When this rate is too fast, one must use a more slowly-growing family depending on a parameter $h$. For instance, for smooth flows on surfaces, the growth rates are $t^h$ and $t(\log t)^h$, depending on whether the singularities of the flow are degenerated or not, respectively \cite{Kanigowski}. For unipotent homogeneous $\R$ actions, the correct family of scales to choose is the polynomial one,  $t^h$ \cite{KanigowskiVinhageWei}. We refer to the survey paper \cite{KanigowskiKatokWei} for a more systematical review of history and results related to slow entropy.

The dynamical systems considered in this paper are abelian groups acting by translation by unipotent elements on finite volume homogeneous spaces. These are natural generalizations of unipotent flows to the higher rank setting. More precisely, suppose $G$ is a connected Lie group with Lie algebra $\mathfrak{g}$ aand $\Gamma\subset G$ is a lattice, we say $\mathfrak{u}\subset\mathfrak{g}$ is an abelian $\ad$-unipotent subalgebra if $\mathfrak{u}$ is abelian, $\dim\mathfrak{u}\geq 1$ and for every $W\in\mathfrak{u}$, $\ad_W$ is a nilpotent element in $\mathfrak{gl}(\mathfrak{g})$. We obtain the precise value of both topological and metric slow entropy in polynomial scale of the translation action by $\exp(\mathfrak{u})$ on $G/\Gamma$ induced by an abelian $\ad$-unipotent subalgebra $\mathfrak{u}$ of $\mathfrak{g}$. To state our main result, we need some basic definitions. We define $\tilde{\mathfrak{g}}_0=\mf g_0 = Z_{\mf g}(\mf u)$, the centralizer of $\mf u$ in $\mf g$. Inductively, we set:
\begin{equation}\label{eq:basg} \tilde{\mf g}_i = \set{ X \in \mf g : \ad_{U_1}\ad_{U_2}\dots \ad_{U_{i+1}}(X) = 0 \mbox{ for all } U_1,\dots,U_{i+1} \in \mf u},
\end{equation}
and choose some $\mf g_i \subset \tilde{\mf g}_i$ complementary to $\tilde{\mf g}_{i-1}$. If $\norm{\cdot}$ is any norm on $\R^k$, define $F_n = \set{ v \in \R^k : \norm{v} \le n}$. Then $F_n$ is a F{\o}lner sequence, and we call $F_n$ a {\it norm-induced sequence}. Our main result is:
\begin{theorem}
\label{thm:main}
If $\mathfrak{u}$ is an abelian $\ad$-unipotent subalgebra of $\mathfrak{g}$, then the metric and topological polynomial slow entropy of a translation action by $U = \exp(\mf u)$ on $G/\Gamma$  with respect to any norm-induced F{\o}lner sequence is given by:
\[ h_{\mathfrak{u}} = \frac{1}{\dim \mf u}\sum_{i=0}^m \dim(\mf g_i)\cdot i, \]
where $m$ is some positive integer only related to $\mathfrak{g}$ and $\mathfrak{u}$.
\end{theorem}

\begin{remark}
It is interesting to ask whether the polynomial slow entropy is always independent of the F{\o}lner sequence. For instance, for an $\R^2$-action, one may consider the sequence $[-n,n] \times [-n^2,n^2]$.
\end{remark}
\begin{remark}
It is also worth to point out that the polynomial slow entropy is independent of the norm that we use to define the F{\o}lner sets, see Proposition  \ref{prop:EquivalenceOfNormFolner} for a more detailed discussion. From now on, we will keep using the maximum norm on $\mathbb{R}^k$, i.e. $\|v\|=\max_{1\leq i\leq k}|v_i|$ for $v=(v_1,\ldots,v_k)\in\mathbb{R}^k$.
\end{remark}

Our results are a generalization of the results  for flows in \cite{KanigowskiVinhageWei}, but several techniques used for flows \emph{cannot} be adapted to this setting. The principal method in \cite{KanigowskiVinhageWei} was to consider the generator of the flow, $U \in \mf g$, and its adjoint action $\ad_U : \mf g \to \mf g$. This is a linear transformation and can therefore be put in {\it Jordan normal form}. Equivalently, since $\ad_U$ is nilpotent and therefore every eigenvalue is 0, one may choose a basis $\{v_1,\dots,v_n\} \subset \mf g$ such that $\ad_U(v_i) = v_{i+1}$ or $\ad_U(v_i) = 0$ for every $i$. Unfortunately, no such structure exists for commuting nilpotent transformations, there are several examples of unipotent group actions for which there is no common basis which puts every element in Jordan normal form\footnote{A very straight forward example is the following. Let $J_1=${\tiny$\left(
                                         \begin{array}{ccc}
                                           0  & 1 &  0 \\
                                           0 & 0 & 1 \\
                                           0 & 0 & 0 \\
                                         \end{array}
                                       \right)
$}. Then $J_1$ and $J_1^2$ are commuting nilpotent matrices but cannot be brought into Jordan normal form simultaneously.}. It is a very interesting problem to find a normal form for commuting nilpotent transformations.

Our solution to the problem is to construct a basis which shares several properties with the basis described above, but may be much weaker. This construction is carried out in Section \ref{sec:polynomialControl}. The difficulty of the construction is to ensure that divergence rates can still be obtained, and that the polynomials which dictate the divergence of orbits are not redundant. We also require a new argument when proving Lemma \ref{lem:connectionOfHammingAndBowen}, that the decay rates of Hamming and Bowen balls coincide (i.e., that the topological and metric slow entropies coincide). The argument in \cite[Lemma 3.6]{KanigowskiVinhageWei} \emph{cannot} be adapted directly because it heavily relies on the order in $\R$: the proof involves choosing the first time that orbits diverge. This style of argument usually does not carry over to the setting of higher rank abelian groups, since the ``first time'' that a condition is met is less precise and usually less useful. Here, we replace that argument with one which still prominently features the (higher dimensional) Brudnyi-Ganzburg inequality (Theorem \ref{thm:BrudnyiGanzburg}), but now crucially uses the Besicovitch Covering Theorem (Theorem \ref{thm:BesicovitchCovering}). It is interesting to note that our argument here can be used to replace the one in \cite{KanigowskiVinhageWei} and gives a more precise relation between Hamming and Bowen balls.

As an application of our results, we obtain a simple formula of the complexity of \emph{abelian horocyclic actions} on finite volume homogeneous spaces of semisimple Lie groups. We assume that $\mf g$ is semisimple and we say that $\mf u \subset \mf g$ is a {\it horocyclic} subalgebra if there exists some element $X \in \mf g$ such that $\mf u$ is the sum of all generalized eigenspaces of $X$ whose corresponding eigenvalue has positive real part. Recall that every semisimple Lie algebra is a direct sum of simple algebras, $\mf g = \bigoplus_i \mf g^i$. For each simple factor $\mathfrak{g}^i$ of $\mathfrak{g}$, let $\pi^i: \mathfrak{g} \to \mathfrak{g}^i$ denote the projection. Say that $\mathfrak{g}^i$ is {\it detected} by $\mathfrak{u}$ if $\pi^i(\mathfrak{u}) \neq\{0\}$, and let $D$ denote the set of indices $i$ for which $\mathfrak{g}^i$ is detected. Then we prove:
\begin{theorem}
\label{thm:full-unstable}
If $\mf u$ is an abelian horocyclic subalgebra of $\mf g$, then the metric and topological polynomial slow entropy of the translation action by $U = \exp(\mf u)$ on $G/\Gamma$ (not necessarily compact) with respect to any norm-induced sequence is
$$\frac{1}{\dim\mathfrak{u}}\sum_{i \in D} \dim \mathfrak{g}^i.$$
Moreover, if $\mathfrak{g}$ is simple, then we have its polynomial slow entropy is
$\dfrac{\dim \mf g}{\dim \mf u}.$
\end{theorem}
Moreover, we also establish some properties of the slow entropy for abelian actions in a general setting. We provide a weighted average formula (Proposition \ref{prop:weightAverage}) to explain the relation of polynomial slow entropies of abelian actions and polynomial slow entropies of their generators. We also show that for an $\mathbb{R}^k$ action the slow entropy in the metric or topological case is zero for all scaling functions if and only if the action is conjugate in the respective category to an $\mathbb{R}^k$ action by translations on a compact abelian group (Proposition \ref{prop:metricVanishing} and Proposition \ref{prop:topologicalVanishing}). This generalizes the characterization of total vanishing of slow entropy from rank one case\footnote{The characterization in measurable category was proven by Ferenczi in \cite{Ferenczi} and in topological category by Kanigowski, Vinhage and Wei in \cite{KanigowskiVinhageWei}.} to the setting of higher rank abelian actions.

\subsection*{Plan of the paper}
In Section \ref{sec:topoSlowDef}--\ref{sec:metricSlowDef}, we provide definitions of topological slow entropy and metric slow entropy. Further basic notations and preliminaries for homogeneous dynamics, adjoint representations, uniformities and entropy, and the independence of polynomial slow entropy from norms are presented in Section \ref{sec:preliminaries}. In Section \ref{sec:mainTopologicalCase}, we build a so-called {\em generalized chain basis} (Definition \ref{def:generalizedChainBasis}) of $\mf g$ and give the proofs of several lemmas controlling the coefficients of associated polynomials. As a result, we obtain the precise value of the topological slow entropy of abelian unipotent actions on finite volume homogeneous spaces. Using the higher dimensional Brudnyi-Ganzburg inequality and Besicovitch covering theorem we show in Section \ref{sec:mainMetricSlow} that in compact homogeneous spaces two points being Hamming close implies that they are Bowen close. This yields the metric slow entropy in this setting. Afterwards, we extend the results of Section \ref{sec:mainMetricSlow} to finite volume homogeneous spaces and, hereby, we obtain in Section \ref{sec:mainMetricSlowNoncompact} the precise value of metric slow entropy also in the non-compact setting. In Section \ref{sec:horocyclicActions} we analyze and utilize the structure of semisimple Lie groups to prove Theorem \ref{thm:full-unstable}. In Section \ref{sec:coherenceExamples}, we first show the coherence of Theorem \ref{thm:main} with the slow entropy formula for unipotent flows in \cite{KanigowskiVinhageWei} and then we provide several examples to demonstrate how to calculate the exact value of slow entropy for abelian unipotent actions in different settings. Finally, we investigate general properties of slow entropy for higher rank abelian actions in the appendix. On the one hand, we prove a weighted average formula for product actions in Appendix \ref{sec:ProductRankOne}. On the other hand, in Appendix \ref{sec:entropyVanishingCriterion} we generalize the characterizations of total vanishing of slow entropy in \cite{Ferenczi} and \cite{KanigowskiVinhageWei} to higher rank abelian actions.
\paragraph{Acknowledgements:} The authors are grateful for Svetlana Katok's warm support and careful advice.

\subsection{Topological slow entropy}\label{sec:topoSlowDef}

Let $\alpha:\R^k \curvearrowright (X,d)$ be an action by uniformly continuous homeomorphisms of a metric space $(X,d)$. If $A \subset \R^k$, the {\it Bowen distance} between a pair of points $x$ and $y$ with respect to $A$ is:

\[ d_B^A(x,y) = \sup_{a \in A} d(a \cdot x, a \cdot y) \]

Since the action is by uniformly continuous homeomorphisms, $d_B^A$ is equivalent to $d$. Let $\Bow^A(x,\ve)$ denote the $A$-Bowen ball of radius $\ve$ around $x$. Fix a compact set $K \subset X$. Let $N_B(A,\ve,K)$ be the minimal number of $A$-Bowen balls of radius $\ve$ required to cover $K$ and $S_B(A,\epsilon,K)$ be the maximal number of $A$-Bowen balls of radius $\epsilon$ which can be placed in $X$ disjointly with centers in $K$. It is clear that
$$S_B(A,\epsilon,K)\geq N_B(A,\ve,K)\geq S_B(A,2\epsilon,K).$$

\begin{definition}\label{def:topologicalSlow}
Fix a family of F{\o}lner sets $F_n$ for $\R^k$, and define the {\it (volume-normalized) topological polynomial upper slow entropy of $\alpha:\R^k \curvearrowright (X,d)$ with respect to $F_n$} by:
\[ \overline{h}_{\operatorname{top}}(\alpha)=\sup_{K} \lim_{\ve \to 0} \limsup_{n \to \infty} \dfrac{\log N_B(F_n,\ve,K)}{\log(\Leb(F_n))}=\sup_{K} \lim_{\ve \to 0} \limsup_{n \to \infty} \dfrac{\log S_B(F_n,\ve,K)}{\log(\Leb(F_n))}. \]
Similarly, we define the (volume-normalized) topological polynomial lower slow entropy $\underline{h}_{\operatorname{top}}(\alpha)$ by replacing $\limsup$ with $\liminf$. If $\overline{h}_{\operatorname{top}}(\alpha)=\underline{h}_{\operatorname{top}}(\alpha)$, then we define this value as the {\it (volume-normalized) topological polynomial slow entropy $h_{\operatorname{top}}(\alpha)$ of $\alpha:\R^k \curvearrowright (X,d)$ with respect to $F_n$}.
\end{definition}

\subsection{Metric slow entropy}\label{sec:metricSlowDef}
Let $\alpha:\R^k \curvearrowright (X,\mu)$ be an action of $\R^k$ by measure-preserving transformations, and $\mc P = \set{P_1,\dots,P_n}$ be a finite partition of $X$. The {\it coding map} for the action is the map $\gamma_{\mc P} : X \times \R^k \to  \set{1,\dots,n}$ defined by $a \cdot x \in P_{\gamma_{\mc P}(x,a)}$ for $(x,a) \in X \times \R^k$. If $A \subset \R^k$ is a set with positive Lebesgue measure, the {\it Hamming distance} between a pair of points $x$ and $y$ with respect to $A$ is:
\[ d_H^{A,\mc P}(x,y) = 1 - \dfrac{\Leb(\set{a \in A : \gamma_{\mc P}(a,x) = \gamma_{\mc P}(a,y)})}{\Leb(A)} = \dfrac{\Leb(\set{a \in A : \gamma_{\mc P}(a,x) \not= \gamma_{\mc P}(a,y)})}{\Leb(A)}.\]
Notice that, a priori, $d_H^A$ is only a pseudometric, but that it satisfies the triangle inequality. We may therefore define {\it Hamming balls}:
\[ B_H^{A,\mc P}(x,\ve) = \set{ y \in X : d_H^{A,\mc P}(x,y) < \ve}.\]
Finally, let $N_H(A,\ve,\mc P)$ denote the minimal cardinality of a set $\set{x_1,\dots,x_N}$ such that $\mu\left(\bigcup_{i=1}^N B_H^{A,\mc P}(x_i,\ve)\right) > 1-\ve$.

\begin{definition}\label{def:metricSlow}
Fix a family of F{\o}lner sets $F_n$ for $\R^k$, and define the {\it (volume-normalized) metric polynomial upper slow entropy of $\alpha:\R^k \curvearrowright (X,\mu)$ with respect to $F_n$} by:
\[ \overline{h}_{\mu}(\alpha)=\sup_{\mc P} \lim_{\ve \to 0} \limsup_{n \to \infty} \dfrac{\log N_H(F_n,\ve,\mc P)}{\log(\Leb(F_n))}.\]
Similarly, we define the (volume-normalized) metric polynomial lower slow entropy $\underline{h}_{\mu}(\alpha)$ by replacing $\limsup$ with $\liminf$. If $\overline{h}_{\mu}(\alpha)=\underline{h}_{\mu}(\alpha)$, then we define this value as the {\it (volume-normalized) metric polynomial slow entropy $h_{\mu}(\alpha)$ of $\alpha:\R^k \curvearrowright (X,d)$ with respect to $F_n$}.
\end{definition}

\section{Preliminaries on homogeneous spaces and algebraic actions}\label{sec:preliminaries}

\subsection{Metrics and measures on homogeneous spaces}\label{sec:MetricsMeasures}
Suppose $G$ is a connected Lie group with Lie algebra $\mathfrak{g}$ and $\Gamma\subset G$ is a discrete subgroup. We consider the metric $d_G$ on $G/\Gamma$ induced by the right invariant metric on $G$. More precisely, fix an inner product $\langle\cdot,\cdot\rangle_0$ on $\mathfrak{g}$, then define $\langle\cdot,\cdot\rangle$ for $v,w\in T_gG$ as
$$\langle v,w\rangle=\langle dR_{g^{-1}}v,dR_{g^{-1}}w\rangle_0,$$
where $R_g$ is the right translations on $G$.

Notice that $\langle\cdot,\cdot\rangle$ is right-invariant by its construction and thus it induces a Riemannian metric on $G/\Gamma$. Recall that the Riemannian metric has an associated $C^{\infty}$ mapping $\exp_{\operatorname{geom}}:\mathfrak{g}\to G$ satisfying
$$d_0\exp_{\operatorname{geom}}=\id.$$
Similar to the algebraic exponential, there is a local inverse of $\exp_{\operatorname{geom}}$ and we denote it as $\log_{\operatorname{geom}}$. The following is a direct consequence of the above construction of the inner product:
\begin{lemma}
The Riemannian volume is a right Haar measure on $G$. In particular, it is independent of the metric $\langle\cdot,\cdot\rangle_0$ when determining a probability measure on the homogeneous space $G/\Gamma$.
\end{lemma}
From now on, we denote the probability measure on  $G/\Gamma$ as $\mu$ and its corresponding Haar measure on $G$ as $\bar{\mu}$.

\subsection{Fundamental domain and injectivity radius}\label{sec:fundamentalDomain}
Suppose that $G$ is a connected Lie group and $\Gamma\subset G$ is a finite volume lattice with canonical projection $\pi:G\to G/\Gamma$. A \emph{Dirichlet fundamental domain} of $G$ for $G/\Gamma$ is a subset $F\subset G$ with $\bar{\mu}(\partial F)=0$ such that
$$F=\{g\in G: d_G(g,e)\leq d(g\gamma,e)\text{ for all }\gamma\in\Gamma\}.$$
By the construction of $F$, we know there exists at least one lift to $F$ for every point of $G/\Gamma$. Moreover, by restriction to an open dense subset of $G/\Gamma$, this lift is unique. It is also worth to point out that $\pi:(F,\bar{\mu})\to (G/\Gamma,\mu)$ is a measurable isomorphism.

For $y\in G/\Gamma$ let $\inj(y)>0$ be the \emph{injectivity radius} of $y$:
$$\inj(y)\doteq\sup\{r\geq 0: B_G(y,r)\cap B_G(y,r)\gamma=\emptyset\text{ for all }\gamma\neq e\},$$
where $B_G(y,r)\subset G$ is the ball centered at $y$ with radius $r$ in the metric $d_G$.

For $K\subset G/\Gamma$ we define
$$\inj(K)=\inf_{y\in K}\inj(y)$$ to be the \emph{injectivity radius} of $K$. We list the following classical lemma for future use:
\begin{lemma}\label{lem:injectiveRadiusComapctSet}
For every $\epsilon>0$, there exists a compact set $K_{\epsilon}\subset F$ with $\bar{\mu}(K_{\epsilon})>1-\epsilon$ such that
$$\inj(\pi(K_{\epsilon}))=\inf_{\gamma\in\Gamma\setminus\{e\}}\inf_{z\in K_{\epsilon}}d_G(z\gamma z^{-1},e)>0.$$
\end{lemma}

\subsection{Uniformities and entropy}\label{sec:Uniformities}

We briefly recall a discussion on uniformities and their use in defining slow entropy in \cite{KanigowskiVinhageWei}. Uniformities allow us to make computations and measure ``distances'' in the Lie algebra rather than the Lie group, where calculations can't be carried out explicitly without great pains. It will do little harm to the reader to skip this section and think of the sets $V^{(\ve)}(x)$ defined here as metric balls $B(x,\ve)$. This formal trick is necessary, since $``d \,\mbox{''}(y,\exp(X)y) = \norm{X}$ does not satisfy the triangle inequality, but can be used to construct a uniformity, as we describe below.

Metric spaces are special examples of topological spaces called {\it uniform spaces}. In a metric space $(X,d)$, the uniformity on $X$ induced by the metric is the collection of open sets of $X\times X$ which contain a set of the form $B^{(\epsilon)} = \set{(x,y) : d(x,y) < \epsilon, x,y \in X} \subset X \times X$. Notice that the balls $B(x,\epsilon)$ are exactly $B^{(\epsilon)}(x) = (\set{x} \times X) \cap B^{(\epsilon)}$. More generally, if $V$ is an element of the uniformity $\mc U$, we let $V(x) = (\set{x} \times X) \cap V$. Uniformities axiomatize certain properties of these subsets; for basics on uniform topological spaces, we refer the reader to \cite[Chapter 6]{kelley}. Let $\mc U$ be the uniformity induced by the metric on $G$ and $G/\Gamma$,  which we abusively denote by the same letter. Recall that if $U,V \in \mc U$, then \[U \, * \, V = \set{(x,z) : \mbox{there exists }y\mbox{ such that }(x,y) \in U, (y,z) \in V}.\]

Let $G$ be a Lie group and $\Gamma$ be a discrete subgroup. Notice that for every $ x=\bar{x}\Gamma \in G/\Gamma$, there exists $\epsilon > 0$ such that the map $X \mapsto \exp(X)x= \exp(X)\bar{x}\Gamma$ is injective on $\set{ X \in \mf g : \norm{X} < \epsilon}$. This follows easily from the discreteness of $\Gamma$, and the fact that $\exp(X_i)x\gamma_i \to \exp(X)x\gamma$ implies that the distance between $x^{-1} \exp(-X)\exp(X_i) x$ and $\gamma\gamma_i^{-1}$ tends to zero. Hence if $X_i$ and $X$ are both sufficiently close to 0 (with the closeness depending on the conjugation action of $x$), $\gamma_i$ is eventually constant, and we may assume without loss of generality it is $\id$. Therefore, $X_i$ converges to $X$ and the map is locally injective on a sufficiently small neighborhood of $x$ (which may vary with $x$).

%Fix some compact set $K \subset \Gamma \backslash G$. Since the size of the neighborhood above depends continuosly on $x$, there exists some $\epsilon > 0$ such that $X \mapsto \bar{x}\exp(X)$ is injective on $\set{ X \in \mf g : \norm{X} < \epsilon}$ for every $\bar{x} \in K$. The supremum over all such $\epsilon$ is called the {\it injectivity radius} of $K$, denoted $\inj(K)$.

Suppose that there is a fixed norm $\norm{\cdot}$ on $\mf g$. The following lemma will be useful  later:

\begin{lemma}
\label{lem:uniformity-base}
If $G$ is a Lie group with Lie algebra $\mf g$, then  for any sufficiently small $\epsilon_0$, the collection $\mc B$ of sets $V^{(\epsilon)} = \set{(g,\exp(X)g) : g \in G, \norm{X} < \epsilon}$, $0 < \epsilon < \epsilon_0$ is a base of the uniformity of $G$ and $G/\Gamma$ induced by any left-invariant metric for any norm $\norm{\cdot}$ on $\mf g$.
\end{lemma}
%Recall the base $\set{V^{(\epsilon)} : \epsilon > 0}$ for the uniformity on $G / \Gamma$ as described in Section \ref{sec:uniformities}.
We will work with the base $\mc B$, quite extensively, and therefore establish convenient notation related to it. If $\bar{x},\bar{y}\in G$ and $ x,y \in G/\Gamma$, then $\bar{y} \in V^{(\epsilon)}(\bar{x})$ implies that $\bar{y} = \exp(X)\bar{x}$ for some $X = X(\bar{x},\bar{y}) \in \mf g$ with $\norm{X} < \epsilon$.  Notice that $X(\bar{x},\bar{y})$ is well-defined for points $\bar{x},\bar{y} \in G$ or $x,y \in G/\Gamma$ if they are sufficiently close.

Given some $V \in \mc U$ and $K \subset X$ compact, we define a $V$-separated set of $K$ to be a set of points $\set{x_i}$ such that $(x_i,x_j) \not\in V$ if $i \not= j$.\footnote{Notice that if $V = \set{ (x,y) \in X \times X : d(x,y) < \epsilon}$, then this is the usual notion of an $\epsilon$-separated set.} We similarly define a $V$-cover. Given a uniformly continuous action $\alpha : \R^k \curvearrowright X$, $A \subset \R^k$ and $V \in \mc U$, let $$V_A = \bigcap_{a \in A} (a \times a)^{-1}(V) \in \mc U.$$ It is worth to notice that if $V$ is the set corresponding to the metric ball as above, then $V_A(x) = V_A \cap (\set{x} \times X) \subset X$ is the set corresponding to the Bowen ball determined by $A$.

Given $K \subset X$ compact, let $N_B(A,V,K)$ be the minimal cardinality of a $V_A$-cover of $X$, and $S_B(A,V,K)$ be the maximal cardinality of a $V_A$-separated subset of $K$.

The following is a simple adaptation of an argument of Hood in \cite{hood74}:

\begin{lemma}
\label{lem:uniform-entropy}
If $V^{(\epsilon)} \in \mc U$ is a nested sequence of subsets which is a base of the uniformity $\mc U$, then one may replace $N_B(A,\ve,K)$ and $S_B(A,\ve,K)$ in the definition of topological slow entropy with $N_B(A,V^{(\ve)},K)$ and $S_B(A,V^{(\ve)},K)$.
\end{lemma}

\subsection{Independence of norm for norm-induced sequences}\label{sec:EquivalenceOfNormFolner}
In this section, we will prove the following proposition:
\begin{proposition}\label{prop:EquivalenceOfNormFolner}
%Let $\alpha:\mathbb{R}^k\curvearrowright G/\Gamma$ and $F_n$ is a norm-induced sequence, then both metric and topological slow entropy of $\alpha$ are independent of norm.
For any action $\alpha : \R^k \curvearrowright (X,\mu)$, the norm-induced polynomial slow entropy is independent of the norm chosen on $\R^k$.
\end{proposition}
%\begin{remark}
%Indeed this proposition is also true by replacing $G/\Gamma$ by any locally compact metric space $X$.
%\end{remark}
\begin{proof}
For a start, we recall that all norms on $\mathbb{R}^k$ are equivalent: suppose $\|\cdot\|_1$ and $\|\cdot\|_2$ are two norms in $\mathbb{R}^k$, then there are positive constants $a,b$ such that $$a\|v\|_1\leq\|v\|_2\leq b\|v\|_1,\text{ }\forall v\in\mathbb{R}^k.$$

As a natural corollary of norm equivalence, we obtain the following relation for $F_n^1=\{v\in\mathbb{R}^k:\|v\|_1\leq n\}$ and $F_n^2=\{v\in\mathbb{R}^k:\|v\|_2\leq n\}$:
$$F_{an}^2\subset F_n^1\subset F_{bn}^2,$$
and thus
\begin{equation}\label{eq:FolnerEquiv1}
F_{[an]}^2\subset F_{an}^2\subset F_n^1\subset F_{bn}^2\subset F_{[bn]+1}^2,
\end{equation}
where $[x]$ represents the maximal integer less equal than $x\in\mathbb{R}$.

Moreover, if we let $\|v\|_2=\max_{1\leq i\leq k}|v_i|$ for $v=(v_1,\ldots,v_k)\in\mathbb{R}^k$, \eqref{eq:FolnerEquiv1} also implies that
$$\Leb(F_{[an]}^2)=(2[an])^k\leq\Leb(F_n^1)\leq(2[bn]+2)^k=\Leb(F_{[bn]+1}^2).$$
Recall that if $x\in\mathbb{R}$ and $x\geq2$, then $\frac{x}{2}\leq x-1<[x]\leq x\leq2x$. Then when $an,bn\geq 2$, the above inequality gives us
\begin{equation}\label{eq:FolnerEquiv2}
(an)^k\leq\Leb(F_{[an]}^2)\leq\Leb(F_n^1)\leq\Leb(F_{[bn]+1}^2)\leq(5bn)^k
\end{equation}

Also notice that \eqref{eq:FolnerEquiv1} implies that for any compact set $K\subset G/\Gamma$ and any finite measurable partition $\mathcal{P}$ of $G/\Gamma$:
\begin{equation}\label{eq:FolnerEquiv3}
\begin{aligned}
&N_B(F_{[an]}^2,\epsilon,K)\leq N_B(F_n^1,\epsilon,K)\leq N_B(F_{[bn]+1}^2,\epsilon,K),\\
&N_H(F_{[an]}^2,\epsilon,\mathcal{P})\leq N_H(F_n^1,\epsilon,\mathcal{P})\leq N_H(F_{[bn]+1}^2,\epsilon,\mathcal{P}).
\end{aligned}
\end{equation}
Combining \eqref{eq:FolnerEquiv2} and \eqref{eq:FolnerEquiv3}, we obtain for $n$ large enough:
\begin{equation}
\begin{aligned}
&\dfrac{\log N_B(F_{[an]}^2,\ve,K)}{\log((\frac{5b}{a})^k\Leb(F_{[an]}^2))}\leq\dfrac{\log N_B(F_n^1,\ve,K)}{\log(\Leb(F_n^1))}\leq\dfrac{\log N_B(F_{[bn]+1}^2,\ve,K)}{\log((\frac{a}{5b})^k\Leb(F_{[bn]+1}^2))},\\
&\dfrac{\log N_H(F_{[an]}^2,\ve,\mathcal{P})}{\log((\frac{5b}{a})^k\Leb(F_{[an]}^2))}\leq\dfrac{\log N_H(F_n^1,\ve,\mathcal{P})}{\log(\Leb(F_n^1))}\leq\dfrac{\log N_H(F_{[bn]+1}^2,\ve,\mathcal{P})}{\log((\frac{a}{5b})^k\Leb(F_{[bn]+1}^2))}.
\end{aligned}
\end{equation}
which implies that for $\alpha:\mathbb{R}^k\curvearrowright G/\Gamma$, the corresponding slow entropies coincide.
%$$h_{\operatorname{top}}(\alpha)=h_{\operatorname{top}}(\alpha),\text{ }h_{\mu}(\alpha)=h_{\mu}(\alpha).$$
The arbitrariness of the norm $\|\cdot\|_1$ gives the proof of the proposition.
\end{proof}

\subsection{The adjoint representation}\label{sec:adjointAction}
Consider the action of $G$ on itself by conjugation $C_g:h\to g^{-1}hg$, where $g,h\in G$. By taking the derivative of $C_g(h)$ at the identity in the coordinate $h$, we obtain the \emph{adjoint representation} of $G$ on $\mathfrak{g}=T_eG$, which is denoted as $\Ad:G\to\operatorname{\mathfrak{g}}$. By taking the derivative of $\Ad$ in the coordinate $g$, we get the adjoint representation $\ad$ of the Lie algebra $\mathfrak{g}$, which coincides with Lie bracket $\ad(X)Y=[X,Y]$ for $X,Y\in\mathfrak{g}$. We document some standard tools from the theory of Lie groups here:
\begin{lemma}\label{lem:BasicLieTheory}
Suppose that $X,Y\in\mathfrak{g}$, then we have
$$\exp(-X)\exp(Y)\exp(X)=\exp(\Ad(\exp(X))Y),$$
$$\exp(\ad(X)):=\sum_{i=0}^{\infty}\frac{\ad(X)^i}{i!}=\Ad(\exp(X)).$$
\end{lemma}

\section{Topological slow entropy of actions on homogeneous spaces}\label{sec:mainTopologicalCase}

\subsection{Controlled polynomial divergence}\label{sec:polynomialControl}
We recall the subalgebras ${\mf g_i}$ (see \eqref{eq:basg}).
\begin{lemma}
\label{lem:nothing-killed}
There exists some $m$ such that $\mf g = \bigoplus_{i=0}^m \mf g_i$. If $X \in \mf g_i$, there exists $U_1,\dots,U_i \in \mf u$ such that $\ad_{U_1}\dots\ad_{U_i}(X) \in \mf g_0 \setminus \set{0}$.
\end{lemma}

\begin{proof}
The claim that $\mf g = \bigoplus_{i=0}^m \mf g_i$ follows from the fact that given any collection of commuting nilpotent transformations on $\R^N$, the composition of any $N$ such transformations must be trivial. To see the second part of the lemma, notice that if $X \in \mf g_i$, then $X \in \tilde{\mf g}_i$, but $X \not\in \tilde{\mf g}_{i-1}$. That $X \in \tilde{\mf g}_i$ implies that $\ad_{U_1}\dots\ad_{U_i}(X) \in \mf g_0$ by definition of $\tilde{\mf g}_i$ and $\mf g_0$. Since $X \not\in \tilde{\mf g}_{i-1}$ implies that $U_1,\dots,U_i$ can be chosen to make the element nonzero.\end{proof}

%If $V$ and $W$ are vector spaces, let $\Hom(V,W)$ denote the vector space of homomorphisms from $V$ to $W$.
If $V$ is a vector space, let $\Sym^k(V)$ denote the $k$-fold symmetric tensor power of $V$. Define the following family of linear maps for $i \ge 1$:
\[
\begin{array}{rcccl}
 \Phi_i & : & \Sym^i(\mf u) \otimes {\mf g_0}^*  &\to  & {\mf g_i}^*\\
 & & U_1 \otimes \dots \otimes  U_i \otimes \psi & \mapsto& \psi \of \ad_{U_1}\dots\ad_{U_i}.
\end{array}
\]
Note that the map is well-defined from the symmetric tensor power because $\ad$ is bilinear and $\mf u$ is abelian.

\begin{lemma}
\label{lem:annihilator}
Suppose that $W \subset V^*$ is a linear subspace. Let $V_1 = \bigcap_{\theta \in W} \ker \theta \subset V$. Then $\dim(W) = \dim(V) - \dim(V_1)$.%., and $V_2 \subset V$ be a subspace complementary to $V_1$. Then $W = V_2^*$.% if and only if for every $v \in V$, there exists $\theta \in W$ such that $\theta(v) \not= 0$.
\end{lemma}

\begin{proof}
This is a standard theorem in linear algebra: notice that $V_1 \subset V = (V^*)^*$ is exactly the annihilator of $W$. The claim is then exactly the statement that for any vector space $B$ and any subspace $A \subset B$, $\dim(A) + \dim(\operatorname{Ann}(A)) = \dim(B)$, where $\operatorname{Ann}(A)$ denotes the annihilator of $A$.

%Choose a basis $\set{v_1,\dots,v_k,v_{k+1},\dots,v_n} \subset V$ of $V$ such that $\set{v_1,\dots,v_k}$ is a basis of $V_1$. Call $V_2 = \vspan\set{v_{k+1},\dots,v_n}$. We claim that the map $\theta \mapsto \theta|_{V_2}$ is an isomorphism from $W$ to $V_2^*$, which proves the lemma. To see the map is injective, suppose that $\theta_1|_{V_2} = \theta_2|_{V_2}$. Then $(\theta_1 - \theta_2)|_{V_2} = 0$, and since $\theta_1,\theta_2 \in W$, $\theta_1 - \theta_2 = 0$ everywhere. That is, $\theta_1 = \theta_2$.

%To see that the map is surjective, it suffices to show that we can find %we may choose $\theta_{k+1},\dots,\theta_n \in W$ which are linearly independent. First, choose any $\theta_{k+1}$ such that $\theta_{k+1}(v_{k+1}) = 0$. This is possible since $v_{k+1} \not\in V_1$. Then we suppose that $\set{v_{k+1},\dots,v_\ell}$ have been chosen such that $ and we find a linearly independent $v_{\ell+1}$.

%We claim that $\dim W = \dim(V) - \dim(V_1)$, and this will imply the theorem by simple dimension counting arguments. We claim that $V^* = W \oplus V_1^*$. Indeed, let $\theta \in V^*$. Then $\theta_1 = \theta|_{V_1} \in V_1^* \subset V^*$, where the embedding $V_1^* \subset V^*$ is determined by the splitting $V = V_1 \oplus V_2$. Then by construction, $\theta - \theta_1$ is trivial on $V_1$. We claim that it must lie in $W$.

\end{proof}

\begin{lemma}
\label{lem:surjective}
The maps $\Phi_i$ are surjective for $i \ge 1$.
\end{lemma}

\begin{proof}
Notice that Lemma \ref{lem:nothing-killed} implies that for every $X \in \mf g_i\setminus \set{0}$, there exists $\theta \in \Im(\Phi_i)$ such that $\theta(X) \not= 0$. That is, $\bigcap_{\theta \in \Im(\Phi_i)} \ker \theta = \set{0}$. By Lemma \ref{lem:annihilator}, $\Im(\Phi_i) = \mf g_i^*$, as claimed.
\end{proof}

Fix a basis $\set{U_1,\dots,U_k} \subset \mf u$ of $U$. Let $\pi_i : \mf g \to \mf g_i$ denote the projection onto $\mf g_i$ determined by the splitting in Lemma \ref{lem:nothing-killed}.%, and $\set{U_1,\dots,U_k}$

\begin{lemma}\label{lem:projectionDegree}
If $U = \sum_{j=1}^k s_jU_j \in \mf u$ and $X \in \mf g_\ell$, % and $\theta \in \mf g_0^*$,
 then $\pi_i( \Ad(\exp(U))X)$ is a polynomial in $s_1,\dots,s_k$ of degree $\max\{\ell -i,0\}$ taking values in $\mf g_i$ whose coefficients depend linearly on $X$.% If $i = 0$, then the polynomial is homogeneous. %Furthermore, $\pi_i( \Ad(\exp(U))X)$ is a homogeneous polyanomial of degree $\ell - i$.
\end{lemma}

\begin{proof}
Notice that $\Ad(\exp(U)) = \exp(\ad_U)=\sum \frac{1}{s!}{\ad_U}^s$, and by definition of $\mf g_\ell$, if $i\leq\ell$ then  ${\ad_U}^{\ell-i}(X) \in \mf g_{i}$. Notice that $\ad_U^{\ell-i} = (s_1\ad_{U_1}+\dots+s_k\ad_{U_k})^{\ell-i}$, which can be expanded with the usual binomial coefficients since $U_1, \dots, U_k$ commute. Therefore, we know that the coefficient of $s_1^{m_1}s_2^{m_2}\dots s_k^{m_k}$ is exactly a multiple (depending only on binomial coefficients) of $\ad_{U_1}^{m_1}\dots \ad_{U_k}^{m_k}(X)$. Projecting these into each coordinate gives the desired functional.
\end{proof}

Let $n_i = \dim(\mf g_i)$, and fix a basis $\theta_1,\dots,\theta_{n_0}$ of $\mf g_0^*$. If $p(s_1,\dots,s_k)$ is a degree $i$ polynomial, recall that it can be written as a sum of homogeneous polynomials $p = \sum_{j=0}^i p^{(j)}$, where each $p^{(j)}$ is a homogeneous polynomial of degree $j$, and that this expression is unique.

\begin{lemma}
\label{lem:good-basis}
Let $U = \sum_{j=1}^k s_j U_j \in \mf u$. For each $i$, there exists a basis $\set{Y_1,\dots,Y_{n_i}} \subset \mf g_i$ and an $n_i$-tuple $(\alpha_1,\dots,\alpha_{n_i})$ such that $1 \le \alpha_j \le n_0$, $p_{j,i} = \theta_{\alpha_j}(\pi_0(\Ad(\exp(U))Y_j))$ is a degree $i$ polynomial in $s_1,\dots,s_k$. Furthermore, for every $\alpha\in \{1,\dots ,n_0\}$, if $I_\alpha = \set{ j : \alpha_j = \alpha}$, then $\set{p_{j,i}^{(i)} : j \in I_\alpha}$ is a linearly independent set of polynomials.
%for all $1 \le j_1 \not= j_2 \le n_i$, either:

%\begin{enumerate}[(a)]
%\item $\alpha_{j_1} \not= \alpha_{j_2}$ or
%\item the polynomial $%the degree $i$ monomials appearing in the polynomials $\theta_{\alpha_{j_1}}(\Ad(\exp(U))Y_{j_1})$ and \\ $\theta_{\alpha_{j_2}}(\Ad(\exp(U))Y_{j_2})$ are not the same.
%\end{enumerate}
\end{lemma}

\begin{proof}
Notice that $\mc B = \set{U_{\ell_1} \otimes \dots \otimes U_{\ell_i} \otimes \theta_\alpha : 1 \le \ell_1 \le \dots \ell_i \le k, 1\le \alpha \le n_0}$ is a basis of $\Sym^i(\mf u) \otimes \mf g_0^*$. Since $\Phi_i$ is surjective, $\Phi_i(\mc B)$ spans $\mf g_i^*$. Therefore, we may choose a subcollection which forms a basis of $\mf g_i^*$. Let $\set{Y_j}$ denote the corresponding dual basis. Then by construction, each element $Y_j$ has an associated $\alpha_j$ which is the last component of the tensor product of the element of $\mc B$. The fact that the degree of $p_{j,i}$ equals $i$ follows from Lemma \ref{lem:projectionDegree}.

Now we prove linear independence. We prove this by considering the homogeneous summand $p_{j,i}^{(i)}$ of each $p_{j,i}$ of  degree $i$. Fix $\alpha$ and consider the polynomials $p_{j,i}^{(i)}$ %$\theta_{\alpha_j}(\Ad(\exp(U))Y_j)$
for which $\alpha_j  =\alpha$. Suppose that there were some linear relations $\sum \tau_j p_{j,i}^{(i)} = 0$ for some $\tau_j \in \R$. Notice that since $\alpha$ is fixed, each $Y_j$ has a unique tensor $U_{\ell_1} \otimes \dots \otimes U_{\ell_i}$ associated to it, giving a unique associated monomial term (since it is a symmetric tensor power). Since $Y_j$ is dual to the images of such elements under $\Phi_i$, each such monomial term can appear at most once in the collection $\set{p_{j,i}^{(i)}}$. This implies that $\tau_j = 0$ for all $j$, and that the polynomials are linearly independent.
\end{proof}

\begin{definition}[Generalized Chain Basis]\label{def:generalizedChainBasis}
A generalized chain basis of $\mathfrak{g}$ is obtained by choosing a basis of $\mf g_i$ as in Lemma \ref{lem:good-basis}. We denote it by
\[\set{Y_{j,i} : 0 \le i \le m, 1 \le j \le n_i}.\]
\end{definition}

 Let $\set{p_{j,i}}$ denote the corresponding set of polynomials. We prove the following abstract lemma:

\begin{lemma}
\label{lem:polynomial-decay}
Let $g_1,\dots,g_\ell$ be a collection of linearly independent homogeneous polynomials in $s_1,\dots,s_k$, where $g_i$ has degree $d_i$, and define $V = \vspan \set{g_1,\dots,g_\ell}$. There exists $C(g_1,\dots,g_\ell) > 0$ such that if $\abs{\sum_{i=1}^\ell x_ig_i(s)} < \ve$ for all $s \in [-R,R]^k$, then $\abs{x_i} < C\ve R^{-d_i}$. Similarly, if $\abs{x_i} < \ve R^{-d_i}$, then $\abs{\sum_{i=1}^\ell x_ig_i(s)} < C\ve$ for all $s \in [-R,R]^k$.
\end{lemma}

\begin{proof}
We define two norms on $V$. Let $g = \sum_{i=1}^n x_i g_i$. The first norm is determined by the coefficients: $\norm{g}_1 = \max \abs{x_i}$. Notice that this is a norm since we have assumed that $g_i$ are linearly independent. The second norm is determined by its supremum on $[-1,1]^k$:
\[ \norm{g}_2 = \sup_{s \in [-1,1]^k} \abs{g(s)}. \]

Notice that $g_2$ is a norm since a polynomial is 0 if and only if it is 0 on an open set. Since any two norms on a finite dimensional space are equivalent, we have that there exists $C > 0$ such that $\norm{g}_1 \le C\norm{g}_2$ and $\norm{g}_2 \le C \norm{g}_1$. Now observe that:
\[ \sup_{s \in [-R,R]^k} \abs{g(s)} = \sup_{s \in [-1,1]^k} \abs{g(Rs)},\]
and that by the assumption of homogeneity, the coefficients of $g(Rs)$ in the polynomials $g_i$ are exactly the original coefficients of $g(s)$ scaled by $R^{d_i}$. This gives the result.
\end{proof}

\begin{corollary}
\label{cor:polynomial-decay}
Let $g_1,\dots,g_\ell$ be a collection of polynomials in $s_1,\dots,s_k$, where $g_i$ has degree $d_i$ and for each $d$ the elements of the set $\{g^{(d_i)}_i\;:\; d_i=d\}$ are linearly independent. Define $V = \vspan \set{g_1,\dots,g_\ell}$. Then there exists $C(g_1,\dots,g_\ell) > 0$ and $R_0 > 0$ such that if $R \ge R_0$, and $\abs{\sum_{i=1}^\ell x_ig_i(s)} < \ve$ for all $s \in [-R,R]^k$, then $\abs{x_i} < C\ve R^{-d_i}$. Similarly, if $\abs{x_i} < \ve R^{-d_i}$, then $\abs{\sum_{i=1}^\ell x_ig_i(s)} < C\ve$ for all $s \in [-R,R]^k$.
\end{corollary}

\begin{proof}
Notice that the corollary is identical to Lemma \ref{lem:polynomial-decay}, having dropped the assumption that $g_i$ is homogeneous, and choosing minimal $R_0$ for which the estimate holds. Instead of the vector space $V$ previously described, consider $\mc P(d)$, the space of all polynomials of degree at most $d = \max\set{d_{i}}$. As before, we put two norms on $\mc P(d)$. In fact, the family of norms is identical to those in the proof of Lemma \ref{lem:polynomial-decay}, where the family of homogeneous polynomials is just taken to be the monomials $s_1^{m_1}\dots s_k^{m_k}$, $\sum m_i \le d$. So we have the norms $\norm{\cdot}_1$ which picks out the largest monomial coefficient, and $\norm{\cdot}_2$, which takes the supremum over $[-1,1]^k$. Define a map from $V$ to $\mc P(d)$ just by inclusion. %This induces a map from $\R^\ell \to \R^{k^2}$ defined by decomposing the polynomial $\sum x_ig_i$ into monomial terms.

Notice that $\mc P(d) = \bigoplus_{i=0}^d E_i$, where $E_i$ is the span of all degree $i$ monomials. That is, $E_i$ is the set of all $i$-homogeneous polynomials. We may similarly decompose $V = \bigoplus_{i=0}^d V_i$, where $V_i$ is the span of each of the polynomials of degree $i$ (this space may be empty). We denote $m_i = \dim(V_i)$. There is an inclusion $\psi : V \to \mc P(d)$. Notice that by our assumptions, $\psi(V_i) \subset \bigoplus_{j=0}^i E_j$. Furthermore, our assumption is exactly that $\pi_i \of \psi|_{V_i}$ is injective, where $\pi_i : \mc P(d) \to E_i$ is the canonical projection.

We reorganize the polynomials $g_i$ as follows: let $\set{h_{i,j} : 1 \le j \le m_i}$ be the collection of such polynomials of degree $i$. Let $p = \sum_{i=0}^d \sum_{j=1}^{m_i} x_{i,j}h_{i,j}$. We claim by induction on $i$ that if $p(s) \le \ve$ for all $s \in [-R,R]^k$, then $\abs{x_{i,j}} \le C R^{-i}$ (where the constant $C$ may be allowed to increase at most finitely many times throughout the induction). Our base case is $i = d$. Notice that since $\pi_i \of \psi|_{V_i}$ is injective, the coefficients of the monomial terms in $E_i$ are independent linear combinations of $\set{x_{i,j} : 1 \le j \le m_i}$. Furthermore, since $d$ is the maximal degree, there are no other $x_{i,j}$ terms appearing as coefficients. So by Lemma \ref{lem:polynomial-decay}, we get exactly that $\abs{x_{i,j}} \le CR^{-i}$.

We we proceed by induction. Assume we have shown the decay rate for every degree in $[i+1, d]$. Then notice that the coefficient of each degree $i$ monomial is some linear combination of the $x_{a,j}$ for $i+1 \le a \le d$ and the $x_{i,j}$, where the $x_{i,j}$ combinations are linearly independent. Notice also by assumption that there are exactly $m_i$ linearly independent terms. Now, we know that $x_{a,j} \le CR^{-a}$, so if $R \ge C^2$, we get that $x_{a,j} \le R^{-a+1/2} \le R^{-i-1/2}$. Therefore, since some independent linear combination plus the faster terms must decay at rate at least $R^{-i}$, the coefficients $x_{i,j}$ themselves must decay at rate at least $R^{-i}$.

The second part of the corollary is trivial as $\deg(g_i)=d_i$ and thus we finish the proof of the corollary.
\end{proof}

%\begin{proof}
%Notice that the corollary is identical to Lemma \ref{lem:polynomial-decay}, having dropped the assumption that $g_i$ is homogeneous. We proceed with two steps of induction. Assume that $d_1 \ge  d_2 \ge \dots \ge d_\ell$ are increasing. Our outer induction is the follwing claim: for all $1 \le j \le d_{\sigma(i)}$, $\abs{x_j} < C \ve R^{i-d_j}$ where $d_{\sigma(i)}$ is the last index for which $d_j \ge i$. We will proceed by induction on $i$, starting from $d_\ell$ and decreasing to 0. The claim is obvious for $i = d_\ell$.

%Suppose we have it for $i-1$, then let $j$ be such that $d_j \ge i$. Then:

%\[ \abs{\sum_{j=1}^\ell x_jg_j(s)}  = \abs{\sum_{j=1}^\ell \sum_{n=0}^{d_j} x_jg_j^{(n)}} = \abs{\sum_{j=1}^\ell \sum_{n=0}^{d_j-1}x_jg_j^{(n)}(s) +  x_jg_j^{(d_j)}(s)} \]

%Then

%\[ \abs{\sum_{j=1}^\ell x_jg_j^{(d_j)}(s)} \le \ve + \abs{\sum_{n=0}^{d_j-1} x_jg_j^{(n)}(s)} \le \ve + \ell d_jR^{d_j-i} \]

%By Lemma \ref{lem:polynomial-decay},

%\end{proof}

\begin{lemma}\label{lem:BowenLieAlgebra}
Let $Y = \sum x_{j,i}Y_{j,i}$ with $x_{j,i} \in \R$. There exists $C_0,R_0 >0$ such that if $R \ge R_0$ and $\norm{\Ad(\exp(U))Y)} \le \ve$ for all $U = \sum s_jU_j$ with $s \in [-R,R]^k$, then $\abs{x_{j,i}} \le C_0\ve R^{-i}$ for every $i,j$. Conversely, if $\abs{x_{j,i}} \le \ve R^{-i}$ for every $i,j$, then we have $\norm{\Ad(\exp(U))Y)} \le C_0\ve$ for all $U = \sum s_jU_j$ with $s \in [-R,R]^k$.
\end{lemma}

\begin{proof}
Notice that if $\norm{\Ad(\exp(U))Y} \le \ve$, then $\norm{\pi_0\Ad(\exp(U))Y} < \ve$. Since all norms are equivalent, there exists $C' > 0$ such that if this occurs, then $\abs{\theta_\alpha(\pi_0(\Ad(\exp(U))Y))} < C'\ve$ for all $\alpha = 1,\dots,n_0$. But due to Lemma \ref{lem:BasicLieTheory}, $\theta_\alpha(\pi_0(\Ad(\exp(U))Y))$ is a sum of polynomials $x_{j,i}p_{j,i}$, as described in Lemma \ref{lem:good-basis}, and $p_{j,i}^{(i)}$ are linearly independent.
%Recall that $p_{j,i}$ is a polynoimal of degree $i$. %We claim that if $R$ is sufficiently large and $\abs{\sum_i\sum_{j \in I_\alpha} x_{j,i}p_{j,i}} < \ve$ on $[-R,R]^k$, then $\abs{\sum_i\sum_{j \in I_\alpha} x_{j,i}p_{j,i}^{(i)}} < 2\ve$ on $[-R,R^k]$. This follows from the fact that the polynomials $p_{j,i}$ are fixed, so by choosing $R$ large enough we can make the highest order terms dominate all smaller order terms.
Now, for fixed $\alpha$, %since $p_{j,i}^{(i)}$ are linearly independent,
we apply Corollary \ref{cor:polynomial-decay} and finish the proof of the first part of the Lemma. Recall that Lemma \ref{lem:projectionDegree} and Definition \ref{def:generalizedChainBasis} imply $\deg(\pi_{\ell}(\Ad(\exp(U))Y_{j,i}))=\max\{i-\ell,0\}$. Combining this with $\abs{x_{j,i}} \le \ve R^{-i}$, we finish the proof of second part of the lemma.
\end{proof}

\subsection{Proof of Theorem \ref{thm:main}: topological case}
Fix $\eta > 0$, $\bar{x}\in G$, $X\in\mathfrak{g}$ and define $V^{(\eta)}(\bar{x}) = \set{\exp(X)\bar{x} : \norm{X} < \eta}$. Let $$\Bow(\bar{x},R,\eta) = \set{ \bar{y}\in G : \Phi_s \cdot \bar{x} \in V^{(\eta)}(\Phi_s \cdot \bar{y}) \mbox{ for all } s \in [-R,R]^k},$$
where $\Phi_s$ is the lift of our abelian unipotent action. Moreover, if $x\in G/\Gamma$ and $\pi:G\to G/\Gamma$ is the canonical projection, then we can define $V^{(\eta)}$ and $\Bow$ on $G/\Gamma$ as follows: $$V^{(\eta)}(x)=\pi(\bigcup_{\bar{x}\in\pi^{-1}x}V^{(\eta)}(\bar{x})),$$
$$\Bow(x,R,\eta) = \set{ y \in G/\Gamma: s \cdot x \in V^{(\eta)}(s \cdot y) \mbox{ for all } s \in [-R,R]^k}.$$

\begin{proposition}\label{prop:BowenEstimate}
Let $K\subset G/\Gamma$ be a compact set. Then for $\eta\in(0,\frac{\inj(K)}{10})$, there exists $C >0$ such that for every $x \in K$:
\[ C^{-1}R^{-h} \le  \mu(\Bow(x,R,\eta)) \le CR^{-h}, \]
where $h=\sum_{i=0}^{m}\dim(\mathfrak{g}_i)\cdot i$.
\end{proposition}
\begin{proof}
Suppose that $x\in K$ and $\bar{x}\in\pi^{-1}x$. Notice that $\eta\in(0,\frac{\inj(K)}{10})$ implies $$\pi(\Bow(\bar{x},R,\eta))=\Bow(x,R,\eta),$$ and thus we have
\begin{equation}\label{eq:BowenLower1}
\bar{\mu}(\Bow(\bar{x},R,\eta))=\mu(\pi(\Bow(\bar{x},R,\eta)))=\mu(\Bow(x,R,\eta)).
\end{equation}

By combining Definition \ref{def:generalizedChainBasis} and Lemma \ref{lem:BowenLieAlgebra}, we obtain that
\begin{equation}\label{eq:BowenLower2}
C^{-1}R^{-h}\leq\bar{\mu}(\Bow(\bar{x},R,\eta))\leq C R^h,
\end{equation}
where $C$ is a constant that only depends on $k$, $\eta$ and $C_0$ in Lemma \ref{lem:BowenLieAlgebra} and $h$ satisfies:
$$h=\sum_{i=0}^{m}\sum_{j=1}^{n_i}i=\sum_{i=0}^m\dim(\mathfrak{g}_i)\cdot i.$$

Combining \eqref{eq:BowenLower1} and \eqref{eq:BowenLower2}, we finish the proof of the Proposition.
\end{proof}

Proposition \ref{prop:BowenEstimate} and Lemma \ref{lem:uniform-entropy} together prove the topological slow entropy case of Theorem \ref{thm:main} though standard arguments (see, e.g., \cite{KanigowskiVinhageWei}). We provide a proof here for completeness.

\begin{corollary}\label{cor:topologicalSlowEntropyCompact}
Let  $\alpha:\mathbb{R}^k\curvearrowright G/\Gamma$ be an abelian unipotent action generated by $\mathfrak{u}$. Then the topological polynomial slow entropy of $\alpha$ is $h_{\mathfrak{u}}$, where $h_{\mathfrak{u}}$ is defined in Theorem \ref{thm:main}.
\end{corollary}
\begin{proof}
By the definition of topological polynomial slow entropy (Definition \ref{def:topologicalSlow}), we can estimate the topological slow entropy either through the maximal number of disjoint balls in $G/\Gamma$ with centers in an arbitrary compact set or the minimal number of balls to cover an arbitrary compact set.

To get an upper bound of the topological polynomial slow entropy of $\alpha$, we fix $K\subset G/\Gamma$ as a compact set and recall that $S_B(F_n,\epsilon,K)$ is the maximal number of $A$-Bowen balls of radius $\epsilon$ which can be placed disjointly in $G/\Gamma$ with centers in $K$. Then Lemma \ref{lem:uniform-entropy} and Proposition \ref{prop:BowenEstimate} imply that
\begin{equation}\label{eq:MaximalSeparatedSets}
S_B(F_n,\epsilon,K)\leq Cn^h,
\end{equation}
where $F_n=[-n,n]^k$ and $h=\sum_{i=0}^{m}\dim(\mathfrak{g}_i)\cdot i$.
Since \eqref{eq:MaximalSeparatedSets} is true for any compact subset of $G/\Gamma$,  $\Leb(F_n)=(2n)^k$ and $k=\dim(\mathfrak{u})$, we obtain that
\begin{equation}\label{eq:topologicalUpperBound}
h_{\operatorname{top}}(\alpha)\leq h_{\mathfrak{u}},
\end{equation}
where $h_{\mathfrak{u}}=\frac{1}{\dim(\mathfrak{u})}\sum_{i=0}^m\dim(\mathfrak{g}_i)\cdot i$.

For the lower bound of topological polynomial slow entropy we fix $K\subset G/\Gamma$ as a compact set again and recall that $N_B(F_n,\epsilon,K)$ is the minimal number of $A-$Bowen balls of radius $\epsilon$ to cover $K$. Then Lemma \ref{lem:uniform-entropy} and Proposition \ref{prop:BowenEstimate} imply that
\begin{equation}\label{eq:minimalCovering}
N_B(F_n,\epsilon,K)\geq C^{-1}n^h.
\end{equation}
Because \eqref{eq:minimalCovering} is true for any compact subset of $G/\Gamma$, we obtain that
\begin{equation}\label{eq:topologicalLowerBound}
h_{\operatorname{top}}(\alpha)\geq h_{\mathfrak{u}}.
\end{equation}

Combining \eqref{eq:topologicalUpperBound} and \eqref{eq:topologicalLowerBound}, we finish the proof of the corollary.
\end{proof}

\section{Metric slow entropy of actions on compact homogeneous spaces}\label{sec:mainMetricSlow}
\begin{definition}[Well-partitionable]\label{def:wellPartitionable}
A metric space $X$ is well-partitionable if it is $\sigma$-compact and for any Borel probability measure $\mu$, compact set $K\subset X$, $\epsilon>0$ and $\delta>0$, there exist $\kappa>0$ and a finite partition $\mathcal{P}$ of $K$ whose atoms have diameter between $\frac{\epsilon}{2}$ and $\epsilon$ and such that $\mu\left(\bigcup_{\xi\in\mathcal{P}}\partial_{\kappa}\xi\right)<\delta$, where
$$\partial_{\kappa}\xi=\{y\in X:B(y,\kappa)\cap\xi\neq\emptyset\text{ but }B(y,\kappa)\nsubseteq\xi\}.$$
\end{definition}
\begin{remark}\label{rem:HomoWellPartition}
Note that any smooth manifold is well-partitionable. Moreover, by Theorem 21.13 in \cite{lee}, the homogeneous space $G/\Gamma$ is a smooth manifold and thus well-partitionable.
\end{remark}

The following is a corollary of Proposition $2$ from \cite{Kat-Tho}. In \cite{Kat-Tho} the authors consider a compact space $X$ but their proof can be easily generalized to the situation that $X$ is well-partitionable.
\begin{theorem}[Slow Goodwyn's Theorem]\label{thm:slowGoodwyn}
Suppose that $\alpha:\mathbb{R}^k\curvearrowright(X,d)$ is an action by uniformly continuous homeomorphisms of a metric space. Then for any invariant measure $\mu$:
$$h_{\mu}(\alpha)\leq h_{\operatorname{top}}(\alpha).$$
\end{theorem}

Recall that a sequence of partitions $\xi_m$ of a standard probability space $(X,\mu)$ with a measure-preserving action $\alpha:\mathbb{R}^k\curvearrowright(X,\mu)$ is \emph{generating} if $\xi_m^{\alpha,\mathbb{R}^k}$ converges to the point partition, where $\xi_m^{\alpha,\mathbb{R}^k}=\bigvee_{0}^n\bigvee_{r\in C_n}r\cdot\xi$ and $C_n=[-n,n]^k$. In fact, a generating sequence of partitions simplifies the calculation of slow entropy:
\begin{proposition}[Proposition $1$ in \cite{Kat-Tho}]\label{prop:generatingSequence}
Let $\xi_m$ be a generating sequence of partitions for the action $\alpha:\mathbb{R}^k\curvearrowright(X,d)$, then
$$h_{\mu}(\alpha)=\lim_{m\to\infty}\lim_{\ve \to 0} \lim_{n \to \infty} \dfrac{\log N_H(F_n,\ve,\xi_m)}{\log(\Leb(F_n))}.$$
\end{proposition}

Before we formulate the crucial lemma that connects the Hamming distance and Bowen distance for abelian unipotent actions, we need to state the following important result about multi-variable polynomials, which will help us to estimate the divergence of the orbits.
\begin{theorem}[Brudnyi-Ganzburg inequality \cite{Brudnyi}, \cite{BrudnyiGanzburg}]\label{thm:BrudnyiGanzburg}
Let $\mathcal{P}_{d,k}(\mathbb{R})\subset\mathbb{R}[x_1,\ldots,x_k]$ denote the space of real polynomials of degree at most $d$ and let $|U|$ denote the Lebesgue measure of $U\subset\mathbb{R}^k$. Assume $V\subset\mathbb{R}^k$ is a bounded convex body and $w\subset V$ is a measurable subset. For every $p\in\mathcal{P}_{d,k}(\mathbb{R})$, we have
\begin{equation}
\sup_V|p|\leq\left(\frac{4k|V|}{|w|}\right)^d\sup_w|p|.
\end{equation}
\end{theorem}

\begin{theorem}[Besicovitch Covering Theorem, Theorem 18.1c in \cite{DiBenedetto}]\label{thm:BesicovitchCovering}
Let $E$ be a bounded subset of $\mathbb{R}^k$ and let $\mathcal{F}$ be a collection of cubes in $\mathbb{R}^k$ with faces parallel to the coordinate planes and such that each $x\in E$ is the center of a nontrivial cube $Q(x)$ belonging to $\mathcal{F}$. Then there exists a countable collection $\{x_n\}$ of points $x_n\in E$, and a corresponding collection of cubes $\{Q(x_n)\}$ in $\mathcal{F}$ such that
\begin{equation}\label{eq:besicovitchEquation}
E\subset\bigcup Q(x_n) \text{ and }\sum\chi_{Q(x_n)}\leq 4^k.
\end{equation}
\end{theorem}

\begin{remark}
The second equation of \eqref{eq:besicovitchEquation} gives that \emph{each} point $x\in\mathbb{R}^k$ is covered by at most $4^k$ cubes out of $\{Q(x_n)\}$. Equivalently, at most $4^k$ those cubes overlap at each given point in $\mathbb{R}^k$.
\end{remark}

%There exist a constant $P(n)\in\mathbb{N}$ such that the following is true: Let $A\subset\mathbb{R}^n$ be bounded. Let $\mathscr{B}$ be a family of closed balls so that each element of $A$ is the center of some ball of $\mathscr{B}$. Then there exist balls $B_1,B_2,\ldots\in\mathscr{B}$ such that $\chi_A\leq\sum_{j=1}^{\infty}\chi_{B_j}\leq P(n)$.

In the following context, we denote $C_R=[-R,R]^k$ and $X_s=\Ad(\exp(U_s))X$ for $X\in\mathfrak{g}$, $U_s=\sum s_jU_j\in\mathfrak{u}$ and $s=(s_1,\ldots,s_k)\in\mathbb{R}^k$.
\begin{lemma}\label{lem:measureEstimates}
Suppose $\eta>0$ and let $U\in\mathfrak{u}$ and $X\in\mathfrak{g}$ be such that $\|X\|<\eta$. If $\|X_s\|=\eta$ for $s\in\partial C_R$ and $\|X_s\|<\eta$ for all $s\in \mathring{C_R}$, then there exists $c_1>0$ such that
$$|\{s\in C_R:\|\Ad(\exp(U_s))X\|\leq c_1\eta\}|<\frac{|C_R|}{10\cdot12^k}.$$
\end{lemma}
\begin{proof}
Let $w:=\{s\in[-R,R]^k:\|\Ad(\exp(U_s))X\|\leq c_1\eta\}$, $V=C_R$ and $\theta_i$, $i=1,\ldots,n_0$ be a basis of $\mathfrak{g}_0^*$ and then denote $p_i(s)=\theta_i(\pi_0(\Ad(\exp(U_s))X))$.

Since the points separate on the boundary, there exists $i_0\in\{1,\ldots,n_0\}$ such that $\sup_{s\in\partial C_R}p_{i_0}(s)\geq\frac{1}{c_2n_0}\eta$ for a uniform constant $c_2>0$. By Theorem \ref{thm:BrudnyiGanzburg} we have
\begin{equation}\label{eq:BG}
\sup_{V}p_{i_0}(s)\leq\left(\frac{4k|V|}{|w|}\right)^d\sup_wp_{i_0}(s).
\end{equation}

Then we obtain by our construction that $\sup_w|p_{i_0}(s)|\leq\sup_w\|X_s\|\leq c_1\eta$ and $$\sup_V p_{i_0}(s)\geq\sup_{s\in\partial C_R}p_{i_0}(s)\geq\frac{1}{c_2n_0}\eta.$$
As a result, we get from equation (\ref{eq:BG}) that
$$\frac{1}{c_2n_0}\eta\leq\left(\frac{4k|C_R|}{|w|}\right)^dc_1\eta.$$
Therefore we obtain that
$$|w|\leq4k(n_0c_1c_2)^{\frac{1}{d}}|C_R|.$$

By choosing $c_1$ sufficiently small we can guarantee that $4k(n_0c_1c_2)^{\frac{1}{d}}<\frac{1}{10\cdot12^k}$ and thus we conclude the statement.
\end{proof}

For $\eta>0$, let $\mathcal{P}_{\eta}$ be a finite partition of $G/\Gamma$ obtained by Remark \ref{rem:HomoWellPartition} with $K=G/\Gamma$, $\epsilon=\eta$ and $\delta=\frac{1}{100}$. By its construction, we know that if $\xi$ is a atom of $\mathcal{P}_{\eta}$, then $$V^{(\frac{\eta}{2})}(z_1)\subset\xi\subset V^{(\eta)}(z_2),$$ where $z_1,z_2\in G/\Gamma$. Notice that $\mathcal{P}_{\eta}$ converges to the decomposition into points as $\eta\to0$ and thus it suffices to compute the slow entropy with respect to a family of partitions $\mathcal{P}_{\eta_n}$ as $\eta_n\to0$. In the following context, we pick $\eta<\frac{\inj(G/\Gamma)}{10}$, $c_1$ as in Lemma \ref{lem:measureEstimates} and compute the slow entropy with respect to $\mathcal{P}_{c_1\eta}$ for some fixed $\eta$. Then the results follows as $\eta\to0$.
\begin{lemma}\label{lem:connectionOfHammingAndBowen}
There exists $\epsilon_0$, $R_0>0$ such that for every $\epsilon<\epsilon_0$ and every $R>R_0$, the following holds: if the Hamming distance satisfies the inequality $\overline{d}^{R}_{\mathcal{P}_{c_1\eta}}(x,y)<\epsilon$, then $y\in \Bow(x,R,2\eta)$.
\end{lemma}
\begin{proof}
Let $M_{\eta}^R(x,y)=\{s\in C_R:s\cdot y\in\overline{V^{(\eta)}(s\cdot x)}\}$, where $C_R=[-R,R]^k$. Then for every $s\in M_{c_1\eta}^R(x,y)$ define $r(s)=\sup\{r\geq0:(s+[-r,r]^k)\subset M_{\eta}^{3R}(x,y)\}$, where $c_1$ is from Lemma \ref{lem:measureEstimates}. Notice $0<c_1<1$ implies that $r(s)>0$ for all $s\in M_{c_1\eta}^R(x,y)$. To finish the proof it is enough to show that there exists $s\in C_R$ such that $r(s)\geq 2R$.

Suppose that for every $s\in M_{c_1\eta}^R(x,y)$ we have $r(s)<2R$. Then for every $s\in M_{c_1\eta}^R(x,y)$ we denote $C(s)=(s+[-r(s),r(s)]^k)$. Recall that  Besicovitch Covering Theorem (Theorem \ref{thm:BesicovitchCovering}) gives us a subcover $\mathscr{B}=\{C(s_i)\}_{i=1}^{N(R)}$ of $\{C(s):s\in M_{c_1\eta}^R(x,y)\}$  such that for every $s$ there are at most $4^k$ elements from $\mathscr{B}$ containing $s$, where $N(R)\in \mathbb{N}\cup \{\infty\}$ depends on $R$.

Recall that $x,y$ are $\epsilon$ Hamming close on $C_R$ with respect to the partition $\mathcal{P}_{c_1\eta}$. By defining $H_R(x,y)=\{s\in C_R:\text{$s\cdot x$ and $s\cdot y$ are in the same atom of $\mathcal{P}_{c_1\eta}$.}\}$, we have
\begin{equation}\label{eq:HammingLarge}
|H_R(x,y)|\geq(1-\epsilon)|C_R|.
\end{equation}

Then by the definition of $M_{c_1\eta}^R(x,y)$ we have
\begin{equation}\label{eq:geometricMeasure1}
H_R(x,y)\subset M_{c_1\eta}^R(x,y).
\end{equation}

However since $\eta<\inj(G/\Gamma)$ and $r(s)<2R$, Lemma \ref{lem:measureEstimates} implies
\begin{equation}\label{eq:geometricMeasure2}
|\{s\in C(s_i):s\cdot y\in \overline{V^{(c_1\eta)}(s\cdot x)}\}|\leq\frac{1}{10\cdot12^k}|C(s_i)|.
\end{equation}

Also recall that Theorem \ref{thm:BesicovitchCovering} and the properties of the subcover $\mathscr{B}$ give
\begin{equation}\label{eq:geometricMeasure3}
\begin{aligned}
&\sum_{i=1}^{N(R)}|C(s_i)|\leq 4^k|C_{3R}|=12^k|C_R|,\\
M_{c_1\eta}^R(x,y)\subset&\bigcup_{i=1}^{N(R)}\{s\in C(s_i):s\cdot y\in \overline{V^{(c_1\eta)}(s\cdot x)}\}.
\end{aligned}
\end{equation}

Combining  \eqref{eq:geometricMeasure1}, \eqref{eq:geometricMeasure2} and \eqref{eq:geometricMeasure3}, we obtain
$$|H_R(x,y)|\leq\frac{1}{10}|C_R|,$$
which contradicts equation \eqref{eq:HammingLarge} and thus we know that there exists $s_0\in M_{c_1\eta}^R(x,y)$ with $r(s_0)\geq 2R$. This finishes the proof.
\end{proof}

\subsection{Proof of Theorem \ref{thm:main}: metric slow entropy in compact case}
With the help of Proposition \ref{prop:BowenEstimate}, Corollary \ref{cor:topologicalSlowEntropyCompact}, Theorem \ref{thm:slowGoodwyn} and Lemma \ref{lem:connectionOfHammingAndBowen}, we can compute the metric slow entropy of abelian actions on compact homogeneous spaces. Suppose $G/\Gamma$ is compact, then Lemma \ref{lem:connectionOfHammingAndBowen} implies that
\begin{equation}\label{eq:BowenContainCompact}
B_H^{C_R,\mathcal{P}_{c_1\eta}}(x,\epsilon)\subset\Bow(x,R,2\eta).
\end{equation}
Combining \eqref{eq:BowenContainCompact} with  Proposition \ref{prop:BowenEstimate}, we have
\begin{equation}\label{eq:HammingEstimateCompact}
\mu(B_H^{C_R,\mathcal{P}_{c_1\eta}}(x,\epsilon))\leq CR^{-h},
\end{equation}
where $h=\sum_{i=0}^{m}\dim(\mathfrak{g}_i)\cdot i$. Notice that $\eqref{eq:HammingEstimateCompact}$ implies that we need at least $C^{-1}R^h$ different $\epsilon-$Hamming balls with respect to $\mathcal{P}_{c_1\eta}$ to cover $G/\Gamma$. As a result, we obtain that if $G/\Gamma$ is compact,
\begin{equation}\label{eq:metricSlowLowerCompact}
h_{\mu}(\alpha)\geq h_{\mathfrak{u}},
\end{equation}
where $F_n=[-n,n]^k$ and $h_{\mathfrak{u}}$ is defined in Theorem \ref{thm:main}.

Combining \eqref{eq:metricSlowLowerCompact}, Corollary \ref{cor:topologicalSlowEntropyCompact} and Theorem \ref{thm:slowGoodwyn}, we obtain that the metric slow entropy of $\alpha:\mathbb{R}^k\curvearrowright G/\Gamma$ is $h_{\mathfrak{u}}$ when $G/\Gamma$ is compact.

\section{Slow entropy of actions on non-compact homogeneous spaces}\label{sec:mainMetricSlowNoncompact}
In this section we will extend our arguments about metric slow entropy of abelian unipotent actions in Section \ref{sec:mainMetricSlow} to noncompact homogeneous spaces.
\subsection{Hamming balls estimates in noncompact homogeneous spaces}
Suppose $G$ is a connected Lie group, $\Gamma\subset G$ is a lattice (not necessarily co-compact) and $\mu$ is the Haar measure on $G/\Gamma$.

\paragraph{Construction of partitions:} For any given $\eta>0$ and $\delta>0$, let $K\subset G/\Gamma$ be a compact subset with $\mu(K)>1-\delta$ and $\overline{\mathcal{P}}_{\eta,\delta}$ be a partition of $K$ obtained by Remark \ref{rem:HomoWellPartition} with coefficients $\eta$ and $\delta$. Then denote $\mathcal{P}_{\eta,\delta}=\overline{\mathcal{P}}_{\eta,\delta}\cup K^c$, which is a finite measurable partition of $G/\Gamma$.

\begin{lemma}\label{lem:connectionOfHammingAndBowenNonCompact}
Suppose $K$ is a compact subset of $G/\Gamma$ such that $\mu(K)>1-\delta$ for some $\delta\in(0,\frac{1}{100})$ and let $\eta\in(0,\frac{\inj(K)}{10})$. Let $\mathcal{P}_{c_1\eta,\delta}$ be defined as above, where $c_1$ is defined in Lemma \ref{lem:measureEstimates}. Then there exist $R_1,\epsilon_0>0$ and a set $L\subset K$ with $\mu(L)>1-4\delta$ such that the following property holds true: for $0<\epsilon<\epsilon_0$, $R\geq R_1$ and $x\in L$, if $d_{\mathcal{P}_{c_1\eta,\delta}}^R(x,y)<\epsilon$, then there exists a lift of $x$ in $G$, denoted as $\bar{x}$, such that  $$\overline{y}\in\Bow(\overline{x},R,2\eta),$$ where $\bar{y}$ is the lift of $y$ minimizing the distance from $\bar{x}$ in $G$.
\end{lemma}
\begin{proof}
Let $L_0(R_1,\epsilon)\subset G/\Gamma$ be the set of all points $x\in G/\Gamma$ such that for all $R\geq R_1$ we have
$$|s\in C_R:s\cdot x\in K|\geq(1-\epsilon)\mu( K)|C_R|.$$
Then by the pointwise ergodic theorem (Theorem 1.3 in \cite{Lindenstrauss}) and definition of well partitionable, for any given $\epsilon>0$, we can find $R_1$ sufficiently large such that $\mu(L_0(R_1,\epsilon))>1-2\delta$. In the next step, we introduce $L=L_0(R_1,\epsilon)\cap K$ and we obtain that $\mu(L)>1-4\delta$.

Let $R\geq R_1$, $x\in L$ and $y$ satisfy $d_{\mathcal{P}_{c_1\eta,\delta}}^R(x,y)<\epsilon$. Let $\overline{s\cdot x},\overline{s\cdot y}\in G$ be the lifts of $s\cdot x,s\cdot y$ to $G$ which minimize their distance in $G$. Define $$\widetilde{M}_{c_1\eta}^R(x,y)=\left\{s\in C_R:s\cdot y\in\overline{V^{(c_1\eta)}(s\cdot x)}\text{ and }s\cdot x\in K\right\},$$
$$\overline{M}_{\eta}^{3R}(x,y)=\{s\in C_{3R}:\Phi_s\cdot \bar{y}\in\overline{V^{(\eta)}(\Phi_s\cdot \bar{x})}\}.$$
Then for every $s\in\widetilde{M}_{c_1\eta}^R(x,y)$,  let $\tilde{r}(s)=\sup\{r\geq0:[-r,r]^k\subset \overline{M}_{\eta}^{3R}(s\cdot x,s\cdot y)\}$, where $c_1$ comes from Lemma \ref{lem:measureEstimates}. It is worth to point out that $\tilde{r}(s)>0$ for all $s\in\widetilde{M}_{c_1\eta}^{R}(x,y)$ as $c_1\in(0,1)$ and $s\cdot x\in K$.

Recall that $\epsilon\in(0,\epsilon_0)$, $x\in L$,
$x,y$ are $\epsilon$-Hamming close for $s\in C_R$ and the atom of $\mathcal{P}_{c_1\eta,\delta}$ other than $K^c$ is contained in $V^{(c_1\eta)}(z)$ for some $z\in K$, thus we may choose $\epsilon_0$ sufficiently small to guarantee that
\begin{equation}\label{eq:HammingLargeNonCompact}
|\widetilde{M}^R_{c_1\eta}(x,y)|\geq\frac{9}{10}|C_R|.
\end{equation}

Suppose that for every $s\in\widetilde{M}_{c_1\eta}^R(x,y)$ we have $\tilde{r}(s)<2R$. Then for every $s\in\widetilde{M}_{c_1\eta}^R(x,y)$ we denote $\widetilde{C}(s)=(s+[-\tilde{r}(s),\tilde{r}(s)]^k)$. Recall that Besicovitch Covering Theorem (Theorem \ref{thm:BesicovitchCovering}) gives us a subcover $\widetilde{\mathscr{B}}=\{\widetilde{C}(\tilde{s}_i)\}_{i=1}^{\widetilde{N}(R)}$ of $\{\widetilde{C}(s):s\in\widetilde{M}_{c_1\eta}^R(x,y)\}$ such that for every $s\in \widetilde{M}_{c_1\eta}^R(x,y)$, there are at most $4^k$ elements from $\widetilde{\mathscr{B}}$ containing $s$, where $N(R)\in \mathbb{N}\cup \{\infty\}$ depends on $R$.

As the injectivity radius of the universal cover is $\infty$, Lemma \ref{lem:measureEstimates} gives us
$$
|\{s\in\widetilde{C}(\tilde{s}_i):d_G\Big(\Phi_{s-\tilde{s}_i}(\overline{\tilde{s}_i\cdot x}),\Phi_{s-\tilde{s}_i} (\overline{\tilde{s}_i\cdot y})\Big)\leq c_1\eta\}|\leq\frac{1}{10\cdot 12^k}|\widetilde{C}(\tilde{s}_i)|.
$$
Notice that for $s\in \widetilde{C}(\tilde{s}_i)$ the points $\Phi_{s-\tilde{s}_i}(\overline{\tilde{s}_i\cdot x})$ and $\Phi_{s-\tilde{s}_i}(\overline{\tilde{s}_i\cdot x})$ are $\eta$ close on the universal cover. Therefore, for every $s\in\widetilde{C}(\tilde{s}_i)$ for which $s\cdot x\in K$, by $\eta\in(0,\frac{\inj(K)}{10})$, we know that
$d_{G/\Gamma}(s\cdot x,s\cdot y)=d_G\Big(\Phi_{s-\tilde{s}_i}(\overline{\tilde{s}_i\cdot x}),\Phi_{s-\tilde{s}_i} (\overline{\tilde{s}_i\cdot y})\Big)$.
This leads to
$$
\{s\in\widetilde{C}(\tilde{s}_i):d_{G/\Gamma}(s\cdot x,s\cdot y)\leq c_1\eta \text{ and }s\cdot x\in K\}\subset $$$$
\{s\in\widetilde{C}(\tilde{s}_i):d_G\Big(\Phi_{s-\tilde{s_i}}(\overline{\tilde{s}_i\cdot x}),\Phi_{ s-\tilde{s}_i} (\overline{\tilde{s}_i\cdot y})\Big)\leq c_1\eta\}.
$$
Therefore,
\begin{equation}\label{eq:geometricMeasurenonCompact1}
|\{s\in\widetilde{C}(\tilde{s}_i):s\cdot y\in\overline{V^{(c_1\eta)}(s\cdot x)}\text{ and }s\cdot x\in K\}|\leq\frac{1}{10\cdot 12^k}|\widetilde{C}(\tilde{s}_i)|.
\end{equation}

Recall that  by the Besicovitch Covering Theorem (Theorem \ref{thm:BesicovitchCovering}) and the definition of $\widetilde{M}_{c_1\eta}^R(x,y)$ we have
\begin{equation}\label{eq:geometricMeasurenonCompact2}
\begin{aligned}
&\sum_{i=1}^{\widetilde{N}(R)}|\widetilde{C}(\tilde{s}_i)|\leq 4^k|C_{3R}|=12^k|C_R|,\\
\widetilde{M}_{c_1\eta}^R(x,y)\subset&\bigcup_{i=1}^{\widetilde{N}(R)}\{s\in\widetilde{C}(\tilde{s}_i):s\cdot y\in\overline{V^{(c_1\eta)}(s\cdot x)}\text{ and }s\cdot x\in K\}.
\end{aligned}
\end{equation}

Combining \eqref{eq:geometricMeasurenonCompact1} and \eqref{eq:geometricMeasurenonCompact2}, we obtain
$$|\widetilde{M}_{c_1\eta}^R(x,y)|\leq{\frac{1}{10}}|C_R|,$$
which contradicts \eqref{eq:HammingLargeNonCompact} and thus we know that there exists $s_0\in \widetilde{M}_{c_1\eta}^R(x,y)$ such that $\tilde{r}(s_0)\geq 2R$. This implies that \begin{equation}\label{eq:s0BowenClose}
\overline{s_0\cdot y}\in\Bow(\overline{s_0\cdot x}, 2R,2\eta),
\end{equation}
which in particular guarantees that $$d_G(\Phi_{-s_0}\cdot\overline{s_0\cdot y},\Phi_{-s_0}\cdot\overline{s_0\cdot x})<2\eta.$$
Then notice by $\Phi_{-s_0}\cdot\overline{s_0\cdot x}\in\pi^{-1}x$,  $\Phi_{-s_0}\cdot\overline{s_0\cdot y}\in\pi^{-1}y$ and $\eta\in(0,\frac{\inj(K)}{10})$ that above inequality implies
\begin{equation}\label{eq:uniquenessLiftY}
V^{(2\eta)}(\Phi_{-s_0}\cdot\overline{s_0\cdot x})\cap\pi^{-1}y=\{\Phi_{-s_0}\cdot\overline{s_0\cdot y}\}.
\end{equation}

By combining \eqref{eq:s0BowenClose}, \eqref{eq:uniquenessLiftY} and $C_R\subset(s_0+C_{2R})$ for $s_0\in C_R$, we finish the proof of the lemma.

\end{proof}

\subsection{Proof of Theorem \ref{thm:main}: noncompact case}
Let $K=\pi(K_{\delta})$, where $K_{\delta}$ comes from Lemma \ref{lem:injectiveRadiusComapctSet} with coefficient $\delta$. Then Lemma \ref{lem:connectionOfHammingAndBowenNonCompact} guarantees that the $\epsilon$-Hamming ball $B_H^{C_R,\mathcal{P}_{c_1\eta,\delta}}(x,\epsilon)$ centered at $x\in L$ with respect to $\mathcal{P}_{c_1\eta,\delta}$ satisfies \begin{equation}\label{eq:HammingBowenContain}
B_H^{C_R,\mathcal{P}_{c_1\eta,\delta}}(x,\epsilon)\subset\pi\left(\bigcup_{\bar{x}\in\pi^{-1}x}\Bow(\bar{x},R,2\eta)\right).
\end{equation}
Recall that $\Gamma$ acts on the right, the abelian unipotent action acts on left and our metric is right invariant. Thus, for any $\bar{x}_1,\bar{x}_2\in\pi^{-1}x$ we have
\begin{equation}\label{eq:projectionLattice}
\pi(\Bow(\bar{x}_1,R,2\eta))=\pi(\Bow(\bar{x}_2,R,2\eta)).
\end{equation}
Then since the injectivity radius for universal cover is $\infty$, Proposition \ref{prop:BowenEstimate} gives
\begin{equation}\label{eq:HammingEstimateNoncompact}
\bar{\mu}(\Bow(\bar{x},R,2\eta))\leq CR^{-h},
\end{equation}
where $\bar{x}\in\pi^{-1}x$ and $h=\sum_{i=0}^{m}\dim(\mathfrak{g}_i)\cdot i$  and $\bar{\mu}$ is the right invariant Haar measure on $G$. Combining \eqref{eq:HammingBowenContain}, \eqref{eq:projectionLattice} and \eqref{eq:HammingEstimateNoncompact}, we have for $x\in L$:
\begin{equation}
\mu(B_H^{C_R,\mathcal{P}_{c_1\eta,\delta}}(x,\epsilon))\leq CR^{-h}.
\end{equation}

Notice that for any $\frac{\epsilon}{3}$-Hamming cover of $L$, we can replace the centers with points in $L$ by replacing the radius $\frac{\epsilon}{3}$ with $\epsilon$. Thus we know that one needs at least $C^{-1}R^h$ different $\epsilon$-Hamming balls with respect to $\mathcal{P}_{c_1\eta,\delta}$ to cover $L$. This implies that one needs at least $C^{-1}R^h$ different $\epsilon$-Hamming balls with respect to $\mathcal{P}_{c_1\eta,\delta}$ to cover $G/\Gamma$ as $L\subset G/\Gamma$. As a result, we obtain that even if $G/\Gamma$ is not necessarily compact,
\begin{equation}\label{eq:metricSlowLowerNonCompact}
h_{\mu}(\alpha)\geq h_{\mathfrak{u}},
\end{equation}
where $F_n=[-n,n]^k$ and $h_{\mathfrak{u}}$ is defined in Theorem \ref{thm:main}.

Recalling Corollary \ref{cor:topologicalSlowEntropyCompact}, we know that the upper bound of topological polynomial slow entropy is $h_{\mathfrak{u}}$ even when $G/\Gamma$ is not compact.

Combining the above estimates with Theorem \ref{thm:slowGoodwyn}, we conclude that both topological and metric polynomial slow entropy of $\alpha:\mathbb{R}^k\curvearrowright G/\Gamma$ are $h_{\mathfrak{u}}$, where $G/\Gamma$ is not necessarily compact.

\section{Abelian horocyclic subalgebras and Proof of Theorem \ref{thm:full-unstable}}\label{sec:horocyclicActions}

Let $G$ be a Lie group with Lie algebra $\mathfrak{g}$ and $X \in \mathfrak{g}$. Recall that $\ad_X : \mathfrak{g} \to \mathfrak{g}$ is an endomorphism of $\mathfrak{g}$. By the standard Jordan normal form of linear transformations, there is a splitting $\mathfrak{g} = \mf g_X^- \oplus \mf g_X^0 \oplus \mf g_X^+$, which is preserved by $\ad_X$ such that the eigenvalues of $\ad_X$ on $\mf g_X^-$, $\mf g_X^0$ and $\mf g_X^+$ have real part negative, zero and positive, respectively.

\begin{definition}
 A subalgebra $\mf u \subset \mf g$ is called {\it horocyclic} if there exists $X \in \mathfrak{g}$ such that $\mf u = \mf g_X^+$. $X$ is called a {\it renormalizing element} of $\mf u$.
\end{definition}

Let us recall a feature of semisimple Lie algebras before proving a structural lemma. A subalgebra $\mf a \subset \mf g$ is called an {\it $\R$-split Cartan subalgebra} if it is an abelian subalgebra such that for every $X \in \mf a$, $\ad_X$ is diagonalized over $\R$ and of maximal dimension. Since $\set{\ad_X : X \in \mf a}$ is a commuting family of linear maps, there exist finitely many functionals $\Delta \subset \mf a^*$ called the set of roots, and associated root spaces, $\mf g_\alpha \subset \mf g$. Then $\mf g$ splits as a vector space $\mf g = \mf a \oplus \mf m \oplus \bigoplus_{\alpha \in \Delta} \mf g_\alpha$ satisfying:
\begin{itemize}
\item $\mf m$ is a compact subalgebra commuting with $\mf a$;
\item if $Y \in \mf g_\alpha$ and $X \in \mf a$, then $[X,Y] = \alpha(X)Y$.
\end{itemize}

%Any semisimple Lie algebra has a Cartan subalgebra, and for any two Cartan subalgebras $\mf a_1, \mf a_2 \subset \mf g$  are conjugate.

%There exists a different decomposition $\mf g = \mf k \oplus \mf p$, called a {\it Cartan decomposition} of $\mf g$, with an associated automorphism $\theta : \mf g \to \mf g$ such that $\mf k$ is a maximal compact subalgebra of $\mf g$, $\mf k$ is the $1$ eigenspace of $\theta$ and $\mf p$ is the $-1$ eigenspace of $\theta$. $\mf p$ contains an $\R$-split Cartan subalgebra, and since any two split Cartan subalgebras are conjugate, we may assume without loss of generality it contains $\mf a$. Therefore, since $\mf p$ is the $-1$, eigenspace of $\theta$, if $X \in \mf a$ and $Y \in \mf g_\alpha$:

%\[ [X,\theta(Y)] = -[\theta(X),\theta(Y)] = -\theta([X,Y]) = -\theta(\alpha(X)Y) = -\alpha(X)\theta(Y) \]

%Hence, $\theta(\mf g_\alpha) = \mf g_{-\alpha}$.

\begin{lemma}
\label{lem:conformal}
Let $\mf g$ be simple. If $\mf u$ is an abelian, horocyclic subalgebra of $\mf g$, then the renormalizing element $X$ can be chosen such that:

\begin{enumerate}[(1)]
\item $X \in \mf a$, where $\mf a\subset \mf g$ is some $\R$-split Cartan subalgebra;
\item $\mf u$ is normalized by $\ad_{\mf a}$;
\item There exists a set of simple roots of $\mf g$ with respect to $\mf a$, $\Delta_s = \set{\alpha_1,\dots,\alpha_n}$, such that $X \in \bigcap_{i=2}^n \ker \alpha_i$;
\item $\dim(\mf g_X^-) = \dim(\mf g_X^+)$;
\item $\ad_X(v) = v$ for all $v \in \mf g_X^+$;
\item $\ad_X(v) = -v$ for all $v \in \mf g_X^-$.
\end{enumerate}
\end{lemma}

\begin{proof}
First, notice that for any element $\bar{X}$ of a semisimple Lie algebra, if $\ad_{\bar{X}} = T + N + R$, where $T$ has real eigenvalues and is diagonalizable, $N$ is nilpotent (as a linear transformation) and $R$ has purely imaginary eigenvalues and is diagonalizable, then there exists $X,U,Z \in \mf g$ such that $\ad_X = T$, $\ad_U = N$ and $\ad_Z = R$. Then $X$ is $\R$-semisimple and $\mf g^+_X = \mf g^+_{\bar{X}}$, so we may without loss of generality assume that $X$ is $\R$-semisimple.

Any $\R$-semisimple element of $\mf g$ belongs to an $\R$-split Cartan subalgebra, so (1) is clear. Since an $\R$-split Cartan subalgebra consists of semisimple elements commuting with $X$, its adjoint action will have a joint diagonalization, % $\ad_{Y}$, where $Y$ is any element in that splitting Cartan subalgebra,
 so (2) is also obvious. From (2), it is clear that $\Lie(U)$ is a sum of root spaces of $\mf a$. Perturb $X$ to a regular element $X'$. Each regular element $X'$ determines a set of simple and positive roots of $\mf g$. If $X'$ is sufficiently close to $X$, then any root $\alpha$ such that $\alpha(X) > 0$ also satisfies $\alpha(X') > 0$. We may also assume that if $\alpha(X') > 0$, then $\alpha(X) \ge 0$. Thus, $\Lie(U) \subset \mf g_{X'}^+$ , and $\mf g_{X'}^+$ is the sum of positive roots. Let $\set{\alpha_1,\dots,\alpha_n}$ be the simple roots for the system of positive roots induced by $X'$.

%Let $\delta$ be the highest root of $\mf g$, so that $\delta = \sum_{i=1}^n c_i\alpha_i$, with $c_i \in \N$.
Notice that $\alpha_i(X) \ge 0$ for $i=1,\dots,n$. Suppose that there are two simple roots which are positive, without loss of generality, let them be $\alpha_1$, $\alpha_2$. By construction, if $\mf g^{\alpha}$ denotes the root space corresponding to $\alpha$, then $\mf g^{\alpha} \subset \Lie(U)$ if and only if $\alpha(X) > 0$. %\textcolor[rgb]{1.00,0.00,0.00}{I am confused with blue sentence as if $\mathfrak{g}^{\alpha}\subset\Lie(U)$, we will have $\alpha(X')>0$ and thus $\alpha(X)\geq0$; the opposite direction seems may not true, for example, we know that each $\alpha_i(X)\geq0$ and then opposite is equivalent to $\Lie(U)=\mathfrak{g}_{X'}^+$, which seems may not be the case.}
 By \cite[Ch. X, Lemma 3.10]{helgason} any positive root can be written as $\sum_{j=1}^\ell \alpha_{i_j}$ for some $\ell \in \N$, where each partial sum is also a root. There exists a highest root $\delta = \sum_{i=1}^n c_i \alpha_i$ with each $c_i \in \N$. In particular, for $\delta$, the sequence $\sum_{j=1}^\ell \alpha_{i_j}$ has every root appearing, including $\alpha_1$ and $\alpha_2$. Once $\alpha_1$ or $\alpha_2$ appears, the partial sums must belong to $\mf g^+_X$ since every simple root is nonnegative on $X$. They each must appear in the sequence, so there exists a root $\beta$ which is a nontrivial partial sum such that either $\beta +\alpha_1$ or $\beta + \alpha_2$ is also a root. Therefore, $\Lie(U)$ is not abelian, since $[\mf g^{\alpha_1},\mf g^\beta] = \mf g^{\alpha_1+\beta}$. Therefore, there exists a unique $\alpha \in \Delta_s$ such that $\alpha(X) > 0$, and it appears first in the sequence $(\alpha_{i_1},\alpha_{i_2},\dots,\alpha_{i_\ell})$, and uniquely in this position. This proves (3).

(4) follows from the fact that if $\alpha$ is a root, so is $-\alpha$, and $\Lie(U)$ is a sum of root spaces. Finally, observe that our arguments above show that a root space $\mf g^{\beta} \subset \Lie(U)$ if and only if $\beta = \alpha_1 + \sum_{i=2}^n c_i\alpha_i$ for $c_i \in \Z_{\ge 0}$. In particular, $\beta(X) = \alpha_1(X)$. Replacing $X$ with $X/\alpha_1(X)$ gives (5) and (6) immediately.
\end{proof}

\begin{lemma}
\label{lem:full-unstable}
If $\mf u = \mf g_X^+$ is an abelian, horocyclic subalgebra of a simple Lie algebra, then $\mf g_0 = \mf g_X^+$, $\mf g_1 = \mf g_X^0$ and $\mf g_2 = \mf g_X^-$, where $\mathfrak{g}_0,\mathfrak{g}_1,\mathfrak{g}_2$ are as in \eqref{eq:basg}.
\end{lemma}

\begin{proof}
To see that $\mf g_0 = \mf g_X^+$, we must show that the centralizer of $\mf g_X^+$ is itself. Since Lemma \ref{lem:conformal} (5) and (6) imply $[\mf g_X^-,\mf g_X^+] \subset \mf g_X^0$ and $[\mf g_X^0,\mf g_X^+] \subset \mf g_X^+$, the centralizer of $\mf u$ must also split as a direct sum. Notice that since any root does not commute with its opposite, the centralizer has trivial intersection on $\mf g_X^-$. Let $\mf z_0 = Z_{\mf g}(\mf u)\cap\mathfrak{g}_X^0$. We therefore wish to show that $\mf z_0 = \set{0}$.

We first claim that $\mf z_0$ is an ideal in $\mf g_0^X$. Indeed, if $W \in \mf g_X^0$, $Z \in \mf z_0$ and $v \in \mf u$, then:
\[ \ad_{[Z,W]}(v) = \ad_Z\ad_Wv - \ad_W\ad_Zv = \ad_ZV' - 0 = 0, \]
where $V' = \ad_Wv \in \mf u$. Therefore, $\mf z_0$ is an ideal in $\mf g_X^0$. $\mf z_0$ acts trivially on $\mf g_X^+$ by definition. %Now observe that if $\theta$ is a Cartan involution fixing $\mf a$, then $\theta(\mf g_X^+) = \mf g_X^-$ and $\theta(\mf g_X^0) = \mf g_X^0$. Therefore, if $\mf z_0' = Z_{\mf g}(\mf g_X^-) \cap \mf g_X^0$, then $\mf z_0' = \theta(\mf z_0)$ it must also act trivially on $\mf g_X^-$.
We claim that $[[\mf g_X^-,\mf z_0],\mf u] \subset \mf z_0$. Indeed, if $V \in \mf g_X^-$, $Y \in \mf z_0$ and $U \in \mf u$, then $H = [V,U] \in \mf g_X^0$ and:
\[ [[V,Y],U] = [V,[Y,U]] - [Y,[V,U]] = [H,Y] \in \mf z_0, \]
since $\mf z_0$ is an ideal in $\mf g_X^0$.

Let $\mf z = \mf z_0 \oplus [\mf g_X^-,\mf z_0]$. We claim that $\mf z$ is an ideal in $\mf g$. Indeed, $\mf z_0$ acts trivially on $\mf u$ and we just showed that $[[\mf g_X^-,\mf z_0], \mf u] \subset \mf z_0$, so $[U,\mf z] \subset \mf z$ for every $U \in \mf u$. If $Y \in \mf g_X^0$, then $[Y,\mf z_0] \subset \mf z_0$ since $\mf z_0$ is an ideal, and since $Y$ preserves both $\mf g_X^-$ and $\mf z_0$, $[Y,[\mf g_X^-, \mf z_0]] \subset [\mf g_X^-, \mf z_0]$. Finally, if $V \in \mf g_X^-$, then by definition, $[V,\mf z_0] \subset [\mf g_X^-,\mf z_0]$ and $[V,[\mf g_X^-,\mf z_0]] \subset [V,\mf g_X^-] = \set{0}$. Therefore $\mf z$ is an ideal in $\mf g$, and notice that $X \not\in \mf z$. Since $\mf g$ is simple, $\mf z = \set{0}$. Since $\mf z = \set{0}$, we know that $\mf z_0 = \set{0}$. This shows that $\mf g_X^+$ is its own centralizer (i.e., that $\mf g_0 = \mf g_X^+$).

%Therefore, $\mf z_0$ is an ideal in $\mf g$, and since it is not all of $\mf g$, it must be trivial.

The other equalities immediately follow from the fact that $[\mf g_X^0,\mf u] \subset \mf u$ and that a root and its opposite always generate a copy of $\mf{sl}(2,\R)$ as a Lie algebra.
\end{proof}

Recall the following terminology defined before the statement of Theorem \ref{thm:full-unstable}: if $\mf g = \bigoplus_i \mf g^i$ is the decomposition of $\mf g$ as a sum of simple subalgebras with projections $\pi^i : \mf g \to \mf g^i$, we say $\mf g^i$ is detected by $\mf u$ if $\pi^i(\mf u) \not= \set{0}$.

\begin{corollary}\label{cor:extendToSemiSimple}
Let $\mf u = \mf g_X^+$ be an abelian, horocyclic subalgebra of a semisimple Lie algebra. With the exception of Lemma \ref{lem:conformal} (3) the conclusions of Lemmas \ref{lem:conformal} and \ref{lem:full-unstable} hold. Instead, one may choose $X$ such that for each simple factor $\mf g^i \subset \mf g$, if $\mf g^i$ is detected, $\pi^i(X)$ is in the kernel of all but one simple root of $\mf g^i$, and if $\mf g^i$ is not detected, then $\pi^i(X) = 0$.
\end{corollary}

\begin{proof}
This follows immediately from the fact that semisimple Lie algebras are direct sums of simple Lie algebras. The projection of $\mf u$ to any simple factor must be of the forms described. Since the adjoint action of one factor is trivial on the others, the normalizing element must project nontrivially to each detected simple factor, i.e. the projection of $\mathfrak{u}$ to this simple factor is nontrivial. This implies that $\mf u$ also splits as a direct sum of abelian, horocyclic subgroups in each factor, and the renormalizing element $X$ can be chosen as the sum of the renormalizing element in each detected factor.
\end{proof}

\subsection{Proof of Theorem \ref{thm:full-unstable}}\label{sec:ProofOfFullUnstable}
\begin{proof}[Proof of Theorem \ref{thm:full-unstable}]

We use Theorem \ref{thm:main}. Let $\mathfrak{g}=\oplus_{i}\mathfrak{g}^i$ be the decomposition of semisimple Lie algebra $\mathfrak{g}$ into direct sum of simple Lie algebras $\mathfrak{g}^i$. Notice that the simple algebras which $\mf u$ does not detect do not contribute any growth. That is, $\mf g^i \subset \mf g_0$. However, if $\mathfrak{g}^i$ is detected, then we have $\dim(\mf g^i) = \dim((\mf g^i)_X^-) + \dim((\mf g^i)_X^0) + \dim((\mf g^i)_X^+)$. So by Lemma \ref{lem:conformal} (4), Lemma \ref{lem:full-unstable} and Corollary \ref{cor:extendToSemiSimple}, the polynomial slow entropy is given by:
\begin{equation}
\begin{aligned}
\sum_{i\in D}\dfrac{0 \cdot \dim (\mf g^i_0) + 1 \cdot \dim(\mf g^i_1) + 2 \cdot \dim(\mf g^i_2)}{\dim \mf u} &= \sum_{i\in D}\dfrac{\dim ((\mf g^i)_X^+) + \dim((\mf g^i)_X^0) + \dim((\mf g^i)_X^-)}{\dim \mf u} \\
&= \sum_{i\in D}\dfrac{\dim \mf g^i}{\dim \mf u},
\end{aligned}
\end{equation}
where $D$ denotes the set of indices $i$ for which $\mathfrak{g}^i$ is detected.
\end{proof}

\section{Coherence with rank one situation and computation of some examples}\label{sec:coherenceExamples}
In this Section we present several applications of our Theorems \ref{thm:main} and \ref{thm:full-unstable}. We start by showing in Section \ref{sec:coherence} that our results conform with the slow entropy computations for unipotent flows in \cite{KanigowskiVinhageWei}. Then we examine several higher rank examples including horocyclic subalgebras of $\mathfrak{sl}(d,\mathbb{R})$ in Section \ref{sec:horocyclic-ex} and restrictions of first row horocyclic subalgebras in Section \ref{sec:restrictionHorocyclic}. We also deal with some examples of abelian unipotent actions on $\mathfrak{sl}(d,\mathbb{R})$ that are not restrictions of first row horocyclic subalgebras (Section \ref{sec:nonrestrictionOfFirstRow}) and abelian unipotent actions on nilpotent homogeneous spaces (Section \ref{sec:abelianNilpotentHomogeneous}).

\subsection{Coherence with rank one situation}\label{sec:coherence}

Suppose $G$ is a connected Lie group with Lie algebra $\mathfrak{g}$ and $U\in\mathfrak{g}$ with $\ad_U^N=0$ for some $N\in\mathbb{Z}^+$. Let $\Gamma\subset G$ be a lattice and denote the unipotent flow induced by $U$ as $\phi_t(g\Gamma)=\exp(Ut)g\Gamma$ for $g\in G$. Since by definition, $\ad_U$ is nilpotent, we may find coordinates of $\mf g$ in which

\begin{equation}
\label{eq:JNF}
\ad_U = \left(
                                  \begin{array}{ccc}
                                    J_{m_1} &  &  \\
                                     & \ddots &  \\
                                     &  & J_{m_n} \\
                                  \end{array}
                                \right),
\end{equation}
 where $J_{m_i}=\left(
                                           \begin{array}{cccc}
                                             0 & 1 &  &  \\
                                              & \ddots & \ddots &  \\
                                              &  & \ddots & 1 \\
                                              &  &  & 0 \\
                                           \end{array}
                                         \right)$ is a $m_i\times m_i$ matrix for $i=1,\ldots,n$. %Based on these notations,
Theorem 1.10 and Theorem 1.11 in \cite{KanigowskiVinhageWei} imply that both topological and metric polynomial slow entropy of $\phi_t$ are equal to
\begin{equation}\label{eq:oldFormula}
\sum_{i=1}^n\binom{m_i}{2}=\sum_{i=1}^n\sum_{j=0}^{m_i-1}j.
\end{equation}

Next we use Theorem \ref{thm:main} to calculate both topological and metric polynomial slow entropy of the unipotent flow $\phi_t$. %Based on Jordan normal form of $\ad_U$,
By \eqref{eq:JNF}, there exists a  basis $\{X_j^i\}$ of $\mathfrak{g}$ with $0\leq j\leq m_i-1$ and $1\leq i\leq n$ such that
$$\ad_U(X_j^i)=X_{j-1}^i$$
for every $j\in\{1,\ldots,m_i-1\}$ and $i\in\{1,\ldots,n\}$.

We now identify some choice of $\mf g_j$ as defined in \eqref{eq:basg}. % and thus we obtain for every integer $j\geq 0$
We claim that we may choose
\begin{equation}\label{eq:rankonegi}
\mathfrak{g}_j=\{X_{j}^i: \text{any $1\leq i\leq n$ such that }j\leq m_i-1\},
\end{equation}
%where $\mathfrak{g}_j$ is defined in Section \ref{sec:mainResults}.

Indeed, one may check directly that $\bigoplus_{\ell=0}^j \mf g_\ell$ is exactly $\tilde{\mf g}_j$.
Then Theorem \ref{thm:main} gives us that both topological and metric polynomial slow entropy are equal to
\begin{equation}\label{eq:newFormula}
\sum_{j=0}^{m}\dim{\mathfrak{g}_j}\cdot j,
\end{equation}
where $m=\max_{i}\{m_i-1\}$.

\vspace{1cm}

\hspace{-.3in}\begin{tabular}{cc}
\begin{minipage}{4.5in}One may easily compare \eqref{eq:newFormula} and \eqref{eq:oldFormula} to see that they are equal: each can be seen to be the sum of $j$ as $X^i_j$ ranges over all basis elements. In the figure to the right, we illustrate the idea for a basis which contains one block of length 5, one of length 4, one of length 2 and one of length 1. Each dot represents a basis element, each block in \eqref{eq:JNF} is represented vertically, with $\ad_U$ taking the basis element to the dot directly above it. One may think of formulas \eqref{eq:newFormula}  and \eqref{eq:oldFormula} as integrating the function $\binom{f(x)}{2}$ where $x$ ranges over the set of blocks and $f(x)$ is their size. Then their different order of summation resembles the difference between Riemann and Lebesgue integration.
\end{minipage}&
$\begin{array}{c|cccc}
4 & \bullet \\
3 & \bullet & \bullet \\
2 & \bullet & \bullet \\
1 & \bullet & \bullet & \bullet & \\
0 & \bullet & \bullet & \bullet & \bullet \\
\hline
 & 1 & 2 & 3 & 4
\end{array}$
\end{tabular}

 %Indeed, \eqref{eq:newFormula} can also be understood as we assign $j$ to each element in $\mathfrak{g}_j$ and then sum them together. However, \eqref{eq:rankonegi} implies that if $j\leq m_i-1$ then we have $X_{j}^i\in\mathfrak{g}_i$ and thus \eqref{eq:newFormula} is also equivalent to assign $j$ for each $X_{j}^i$ if $j\leq m_i-1$. As a result, we obtain the following diagram for $0\leq j\leq m_i-1$ and $1\leq i\leq n$:
%\begin{equation}\label{eq:assignAmount}
%\overset{\overset{m_i-1}{\xdownarrow\quad}}{X_{m_i-1}^i}\overset{\ad_U}{\to}\overset{\overset{m_i-2}{\xdownarrow\quad}}{X_{m_i-2}^i}\overset{\ad_U}{\to}\ldots\overset{\ad_U}{\to}\overset{\overset{j}{\xdownarrow\quad}}{X_j^i}\overset{\ad_U}{\to}\ldots\overset{\ad_U}{\to}\overset{\overset{1}{\xdownarrow\quad}}{X_1^i}\overset{\ad_U}{\to}\overset{\overset{0}{\xdownarrow\quad}}{X_0^i}
%\end{equation}
%By summing together all assigning amounts on above chains over $i$ and $j$, we obtain \eqref{eq:oldFormula} and thus finish our argument about the coherence of two formulas.

%\begin{remark}
%Indeed the difference between counting methods in Theorem \ref{thm:main} and Theorem 1.10, 1.11 in \cite{KanigowskiVinhageWei} are similar to the difference between summing over all elements in a matrix together by rows or columns.
%\end{remark}

\subsection{Horocyclic subalgebras of $\mf{sl}(d,\R)$}
\label{sec:horocyclic-ex}

Let us apply the results of Section \ref{sec:horocyclicActions} to a class of examples on $\mf{sl}(d,\R)$. Notice that since all $\R$-split Cartan subalgebras of $\mf{sl}(d,\R)$ are conjugate, as are any systems of positive roots. Then by Lemma \ref{lem:conformal}, every horocyclic subalgebra of $\mf{sl}(d,\R)$ has an associate simple root. The standard $\R$-split Cartan subalgebra of $\mf{sl}(d,\R)$ are the diagonal matrices $\mf a = \set{\diag(t_t,\dots,t_d) : \sum t_i = 0}$, and the standard choice of simple roots are $\alpha_i(t) = t_i - t_{i+1}$ in these coordinates.

Each $\alpha_i$ has a corresponding horocyclic subalgebra: one chooses $X_i \in \bigcap_{j \not=i} \ker \alpha_j$, and lets $\mf u = \mf g^+_{X_i}$. One may check directly that $X_i = (d-i,\dots,d-i,-i,\dots,-i)$, where $d-i$ is repeated $i$ times, and $-i$ is repeated $d-i$ times. Then $\mf g^+_{X_i}$ is is the set of block diagonal matrices $\begin{pmatrix} \mbf{0}_{i,i} & A \\ \mbf{0}_{d-i,i} & \mbf{0}_{d-i,d-i} \end{pmatrix}$, where $A$ is any $i \times (d-i)$ matrix, and $\mbf{0}_{m,n}$ is the $m \times n$ matrix with all zero entries. Thus, when $i = 1$, the algebra $\mf u$ is the first row of $\mf{sl}(d,\R)$. This is the simplest of the actions, and we refer to it as a {\it first row action}.

It is also straightforward to check that $\mf g^-_{X_i}$ is the corresponding block-lower triangular matrices. By Theorem \ref{thm:full-unstable}, the slow entropy of each of these actions is given by \[ \frac{\dim(\mf{sl}(d,\R))}{\dim \mf u} = \frac{d^2-1}{i(d-i)}.\]

\subsection{Restrictions of first row horocyclic subalgebras}\label{sec:restrictionHorocyclic}

There is one simple technique one may use to produce examples of abelian unipotent actions that are not horocyclic, starting from a horocyclic one: take a restriction of the action. Recall the first row action on $\mf{sl}(d,\R)$ from Section \ref{sec:horocyclic-ex}. Denote $E_{i,j}$ as $d\times d$ matrix with $1$ at position $(i,j)$ and all others are zeros.  We consider the action of the following algebra, which is a restriction of this action: $\mathfrak{u}_\ell = \operatorname{span}_\R\set{E_{1,j} : 2 \le j \le \ell+1}\subset\mathfrak{sl}(d,\mathbb{R})$, where $1\leq\ell\leq d-1$. Notice that when $\ell+1 = d$, this is exactly the first row action.

We remark that the first row action is normalized by $\mf{sl}(d-1,\R)$ sitting in the lower right hand block of $\mf{sl}(d,\R)$, and that the action of $\mf{sl}(d-1,\R)$ is transitive on every Grassmanian of the first row. Thus, these examples represent all examples of first row restrictions up to conjugacy.

By direct computation, we obtain that
\begin{equation}
\begin{aligned}
\mathfrak{g}_0=&\operatorname{span}_\R\left\{\{E_{1,j}:2\leq j\leq d\}\cup\{E_{i,j}:\ell+2\leq i\leq d, 2\leq j\leq d\}\right\},\\
\mathfrak{g}_1=&\operatorname{span}_\R\{\{E_{i,j}:2\leq i\neq j\leq \ell+1\}\cup\{E_{1,1}- E_{j,j}:2\leq j\leq \ell+1\}\\&\cup\{E_{i,1}:\ell+2\leq i\leq d\}
\cup\{E_{i,j}:2\leq i\leq\ell+1,\ell+2\leq j\leq d\}\},\\
\mathfrak{g}_2=&\operatorname{span}_\R\left\{\{E_{i,1}:2\leq i\leq\ell+1\}\right\},\\
\mathfrak{g}_3=&\set{0}.
\end{aligned}
\end{equation}

Moreover, we also obtain the dimension of $\mathfrak{g}_0$, $\mathfrak{g}_1$ and $\mathfrak{g}_2$ as the following:
\begin{equation}
\begin{aligned}
\dim\mathfrak{g}_0&=d-1+(d-\ell-1)(d-1)=(d-\ell)(d-1),\\
\dim\mathfrak{g}_1&=(\ell^2-\ell)+\ell+(d-\ell-1)+\ell(d-\ell-1)=(\ell+1)d-2\ell-1,\\
\dim\mathfrak{g}_2&=\ell.
\end{aligned}
\end{equation}
Finally by applying Theorem \ref{thm:main}, we obtain that both topological and metric polynomial slow entropy of this abelian unipotent action is
$$\frac{1}{\dim\mathfrak{u}}(\dim\mathfrak{g}_0\cdot0+\dim\mathfrak{g}_1\cdot1+\dim\mathfrak{g}_2\cdot2)=\frac{(\ell+1)d-1}{\ell}.$$

It is also worth to point out that based on Theorem \ref{thm:full-unstable}, the abelian unipotent action under our consideration is a horocyclic action if and only if $\ell=d-1$.
\subsection{Non-restrictions of first row horocyclic subalgebras}\label{sec:nonrestrictionOfFirstRow}
In this section we consider some abelian unipotent actions on $\mathfrak{sl}(d,\mathbb{R})$ different from restrictions of the first row action. These type of actions may have nontrivial $\mathfrak{g}_i$ for $i\geq 3$. The example under our consideration in this section is based on an abelian $\ad$-unipotent subalgebra of $\mathfrak{sl}(d,\mathbb{R})$ for $d\geq3$ (as they are horocycle flows if $d=2$ and identity if $d=1$.). More precisely, the subalgebra $\mathfrak{u}$ is defined as:
$$\mathfrak{u}=\operatorname{span}_{\mathbb{R}}\{A_1,\ldots,A_{d-1}\},$$
where $A_k=(a_{i,j})_{d\times d}\in\mathfrak{sl}(d,\mathbb{R})$ satisfies $a_{1,1+k}=a_{2,2+k}=\ldots=a_{d-k,d}=1$ and all other $a_{i,j}=0$ for $1\leq k\leq d-1$. Notice that $A_k=A_{1}^k$ and thus $\mathfrak{u}$ is an abelian $\ad$-unipotent subalgebra of $\mathfrak{sl}(d,\mathbb{R})$.

Recalling Corollary 6.2 (i)\footnote{See page 22, line 7 in \cite{KanigowskiVinhageWei} for more details.} in \cite{KanigowskiVinhageWei}, we know that the chain basis of $\mathfrak{sl}(d,\mathbb{R})$ with respect to $\ad_{A_1}$ consists of $d-1$ single chains with length $3,5,\ldots,2d-1$. Denote this chain basis as $\{X_j^i\}$ for $1\leq i\leq d-1$ and $0\leq j\leq 2d-2i$. By definition of chain basis, we have:
\begin{equation}\label{eq:chainbasis}
\ad_{A_1}X_j^i=X_{j-1}^i,
\end{equation}
for $1\leq i\leq d-1$ and $1\leq j\leq 2d-2i$. Moreover, for $0\leq j\leq 2d-2$, we define: $$H_j=\operatorname{span}_{\mathbb{R}}\{X_{j}^i: \text{any $1\leq i\leq d-1$ such that }j\leq 2d-2i\}.$$

\begin{lemma}\label{lem:hjgjequal}
For $0\leq j\leq 2d-2$, we have $H_j=\mathfrak{g}_j$ and $\mathfrak{sl}(d,\mathbb{R})=\bigoplus_{j=0}^{2d-2}\mathfrak{g}_j$.
\end{lemma}
\begin{proof}
If $V\in H_0$, this implies $[A_1,V]=0$, i.e. $A_1V=VA_1$ and thus we also have $A_1^kV=A_1^{k-1}VA_1=\ldots=VA_1^k$, which implies that $[A_1^k,V]=0$ for $1\leq k\leq d-1$. As a result, we obtain that $H_0\subset\mathfrak{g}_0$. Moreover, if there exists $W\in\mathfrak{g}_0\setminus H_0$, \eqref{eq:chainbasis} implies that $\ad_{A_1}W\neq0$, which contradicts the definition of $\mathfrak{g}_0$. Thus we obtain that $H_0=\mathfrak{g}_0$.

Now suppose $\mathfrak{g}_i=H_i$ for all $0\leq i\leq i_0$. If $V\in H_{i_0+1}$, then for $1\leq k\leq d-1$, since $A_1$ commutes with $A_1^{k}$, we have: $$\ad_{A_1}^{i_0+1}\ad_{A_1^{k}}(V)=\ad_{A_1^{k}}\ad_{A_1}^{i_0+1}(V)\subset\ad_{A_1^{k}}(H_0)=\ad_{A_1^{k}}(\mathfrak{g}_0)=\{0\},$$
which implies that $\ad_{A_1^{k}}(V)\in \bigoplus_{i=0}^{i_0}H_{i}=\bigoplus_{i=0}^{i_0}\mathfrak{g}_{i}$ due to equation \eqref{eq:chainbasis} for the chain basis and induction assumption. Thus we obtain that for $1\leq k_1,\ldots,k_{i_0+2}\leq d-1$,
$$\ad_{A_1^{k_1}}\ldots\ad_{A_1^{k_{i_0+2}}}(H_{i_0+1})\subset\ad_{A_1^{k_1}}\ldots\ad_{A_1^{k_{i_0+1}}}(\bigoplus_{i=0}^{i_0}\mathfrak{g}_{i})=\{0\},$$
which gives us that $H_{i_0+1}\subset\mathfrak{g}_{i_0+1}$ by combining with \eqref{eq:chainbasis}. Moreover, if there is  $W\in\mathfrak{g}_{i_0+1}\setminus H_{i_0+1}$, then \eqref{eq:chainbasis}, $\mathfrak{g}_i=H_i$ for $1\leq i\leq i_0$ and $\mathfrak{g}_{i}\cap\mathfrak{g}_j=\{0\}$ for $i\neq j$ imply that $\ad_{A_1}^{i_0+2}W\neq0$, which contradicts the definition of $\mathfrak{g}_{i_0+1}$. Thus we obtain that $H_{i_0+1}=\mathfrak{g}_{i_0+1}$.

As a result of this induction process, we obtain that for $0\leq j\leq 2d-2$:
\begin{equation}\label{eq:gjhjequal}
\mathfrak{g}_{j}=H_j=\operatorname{span}_{\mathbb{R}}\{X_{j}^i: \text{any $1\leq i\leq d-1$ such that }j\leq 2d-2i\}.
\end{equation}
Combining \eqref{eq:gjhjequal} with $\mathfrak{sl}(d,\mathbb{R})=\bigoplus_{j=0}^{2d-2}H_j$, we finish the proof of the lemma.
\end{proof}

With the help of Lemma \ref{lem:hjgjequal} and by direct computation, we obtain that the polynomial slow entropies of $\exp(\mathfrak{u})$ are
%\begin{equation}
%\frac{1}{d-1}\sum_{i=0}^{2d-2}i\cdot\dim\mathfrak{g}_i=\frac{1}{d-1}\sum_{i=1}^{d-1}\binom{2d-2i+1}{2}=\frac{d(4d+1)}{6}.
%\end{equation}
%{\blue where first equality follows from equivalence of \eqref{eq:oldFormula} and \eqref{eq:newFormula}.}
\begin{equation}
\frac{1}{d-1}\sum_{i=0}^{2d-2}i\cdot\dim\mathfrak{g}_i=\frac{1}{d-1}\sum_{i=1}^{d-1}(d-i)(4i-1)=\frac{d(4d+1)}{6}
\end{equation}
since $\dim (\mathfrak{g}_{2i})=\dim(\mathfrak{g}_{2i-1})=d-i$ for $1\leq i\leq d-1$.

Comparing with Theorem \ref{thm:full-unstable} this also shows that $\mathfrak{u}$ is not a horocyclic subalgebra of $\mathfrak{sl}(d,\mathbb{R})$ for $d\geq 3$.

\subsection{Abelian unipotent actions on nilpotent homogeneous spaces}\label{sec:abelianNilpotentHomogeneous}
Let $\mathfrak{g}$ be the subalgebra of $\mathfrak{sl}(d,\mathbb{R})$ consisting of \emph{strictly upper triangular} elements and for $1\leq \ell\leq d-1$ let $\mathfrak{u}_{\ell}$ be its first row's $\ell$ dimensional abelian  $\ad$-unipotent subalgebra of $\mathfrak{g}$, i.e.
$$\mathfrak{u}_{\ell}=\left\{\left(
                \begin{array}{cc}
                  0_{1,1} & A(t_1,\ldots,t_{\ell}) \\
                  0_{d-1,1} & 0_{d-1,d-1}\\
                \end{array}
              \right):t_1,\ldots,t_{\ell}\in\mathbb{R}\right\},$$
where $A(t_1,\ldots,t_{\ell})=(t_1,\ldots,t_{\ell},0,\ldots,0)\in\mathbb{R}^{d-1}$.

Then by direct computation, we have
\begin{equation}
\begin{aligned}
\mathfrak{g}_0&=\operatorname{span}_{\mathbb{R}}\{E_{1,i}:2\leq i\leq d\}\cup\{E_{i,j}:\ell+2\leq i<j\leq d\},\\
\mathfrak{g}_1&=\operatorname{span}_{\mathbb{R}}\{E_{i,j}:2\leq i<j\leq\ell+1\}\cup\{E_{i,j}:2\leq i\leq \ell+1,\ell+2\leq j\leq d\},\\
\mathfrak{g}_3&=\set{0}.
\end{aligned}
\end{equation}
And thus we have
\begin{equation}\label{eq:nilpotentLiealgebra}
\begin{aligned}
\dim\mathfrak{g}_0&=d-1+\frac{(d-\ell-2)(d-\ell-1)}{2},\\
\dim\mathfrak{g}_1&=\frac{\ell(\ell-1)}{2}+\ell(d-\ell-1).\\
\end{aligned}
\end{equation}

Applying Theorem \ref{thm:main} with $\mathfrak{u}_{\ell}$ and \eqref{eq:nilpotentLiealgebra}, we obtain that the slow entropies of these actions are equal to
$$\frac{(\ell-1)}{2}+(d-\ell-1)=\frac{2d-\ell-3}{2},$$
for $1\leq\ell\leq d-1$.

%{\color{blue}
%\begin{remark}
%There is an interesting existence question for flows in general: given any finite list of positive integers $d_0,d_1,\dots,d_n$, does there an $\ad$-unipotent homogeneous $\R^k$-action such that $\dim(\mf g_i) = d_i$ for $i = 0,\dots,n$ and $\mf g_i = \set{0}$ for all $i > n$? Since semisimple Lie groups have a more rigid structure, a solution to this problem likely lies in nilpotent homogeneous spaces. It is easy to construct a solution for flows by suspending affine transformations of tori, but this trick
%\end{remark}
%}

\appendix
\section{Slow entropy of product actions}\label{sec:ProductRankOne}
Recall that a general method to construct a higher rank abelian action on homogeneous space $G/\Gamma$ is using a family of commuting generators on $G/\Gamma$. However, when $G/\Gamma$ is a product space of homogeneous spaces $G_i/\Gamma_i$ for $1\leq i\leq n$, we can construct an abelian unipotent action on $G/\Gamma$ by considering the unipotent flow on $G_i/\Gamma_i$ as a generator for the action. Indeed, there are some simple but nice relations among the polynomial slow entropies of general abelian actions and the polynomial slow entropies of their generators.

Suppose that for $1\leq i\leq \ell$ we have an $\R^{k_i}$ action $\varphi_i$ on $(X_i,\mu_i,d_i)$ with metric (topological) slow entropy $h_i$. Then let $N = \sum_{i=1}^{\ell} k_i$ and $\alpha$ be the product action of $\R^N$ on $(\prod_i X_i,\prod_i \mu_i,\prod_i d_i)$ given by
\[ \alpha : (t_1,\dots,t_{\ell}) \cdot (x_1,\dots,x_{\ell}) = ((\varphi_1)_{t_1}(x_1),\dots,(\varphi_{\ell})_{t_{\ell}}(x_{\ell})) ,\]
where $t_i \in \R^{k_i}$ and $x_i \in X_i$. Then we have the following weighted average formula for the slow entropies at polynomial scale:
	\begin{proposition}\label{prop:weightAverage}
		The volume-normalized metric (topological) polynomial slow entropy (with respect to any norm-induced F{\o}lner sequence) of $\alpha$ is:
		\[ \frac{1}{N}\sum_{i=1}^{\ell}k_ih_i \]
		where $h_i$ is the metric (topological) polynomial slow entropy of $\varphi_i$.
	\end{proposition}
\begin{proof}
		Let us start with the topological polynomial slow entropy. Let $F_n=[-n,n]^N$. Then $y\in \Bow^{F_N}(x,\epsilon)$ is equivalent to $\tilde{d}(s\cdot x, s\cdot y)<\epsilon$ for every $s\in [-n,n]^N$ which is equivalent to $d_i((\varphi_i)_tx_i,(\varphi_i)_ty_i)<\epsilon$ for every $i$ and $t\in[-n,n]^{k_i}$, where $\tilde{d}=\prod_i d_i$. So the number of Bowen balls for $\alpha$ is the product of the number of Bowen balls for each $\varphi_i$ and so it becomes a sum after taking $\log$. This finishes the proof in the topological category.
		
		Let $\mathcal{P}_i$ be a generating partition for $\varphi_i$. Then $\mathcal{P}=\prod_i \mathcal{P}_i$ is generating for $\alpha$. Let $F_n=[-n,n]^N$. Then $t\cdot x=(t_1,\ldots,t_{\ell}) \cdot x$ and $t\cdot y$ are in one atom of $\mathcal{P}$ if and only if $(\varphi_i)_{t_i}x_i$ and $(\varphi_i)_{t_i}y_i$ are in one atom of $\mathcal{P}_i$ for every $i\leq \ell$. It follows by Fubini's theorem that if $y\in B_H^{F_n,\mathcal{P}}(x,\epsilon)$, then for every $i\leq \ell $, $y_i\in B_H^{[-n,n]^{k_i},\mathcal{P}_i}(x_i,\epsilon)$. Conversely, by Fubini's theorem if $y_i\in B_H^{[-n,n]^{k_i},\mathcal{P}_i}(x_i,\epsilon)$ for $i\leq \ell$, then $y\in B_H^{F_n,\mathcal{P}}(x,\ell \epsilon)$. Since $\ell \epsilon\to 0$ as $\epsilon\to 0$, we get the result analogously to the topological case.
\end{proof}

\begin{remark}
Indeed Proposition \ref{prop:weightAverage} also works for general slow entropy normalized by volume if the scaling functions satisfies $a_T(\chi_1+\chi_2)=a_T(\chi_1)a_T(\chi_2)$ and $a_{T^k}(\chi)=a_T(k\chi)$. For more details we refer to Proposition 4.1.8 in \cite{KanigowskiKatokWei}.
\end{remark}
\begin{comment}
From this we immediately get the following application as indicated above: let $\mf g_1,\dots,\mf g_k$ be Lie algebras of Lie groups $G_1,\dots,G_k$, $U_1,\dots,U_k$ be $\ad$-unipotent elements contained in each $G_1,\dots,G_k$, respectively, and $\Gamma_1,\dots,\Gamma_k$ be corresponding lattices. Then there are unipotent flows $\varphi_1,\dots,\varphi_k$ on $G_1/\Gamma_1,\dots,G_k/\Gamma_k$, respectively. Let
\[ \alpha : (t_1,\dots,t_k) \cdot (x_1,\dots,x_k) = ((\varphi_1)_{t_1}(x_1),\dots,(\varphi_k)_{t_k}(x_k)). \]
\begin{corollary} The volume-normalized metric (topological) slow entropy of $\alpha$ is:
\[ \frac{1}{k} \sum_{i=1}^k h_i, \]
where $h_i$ is the metric (topological) slow entropy of $\varphi_i$.
\end{corollary}
\end{comment}

\section{Criterion of total vanishing of slow entropy}\label{sec:entropyVanishingCriterion}
In this section, we generalize Ferenczi's (\cite{Ferenczi}) criterion of total vanishing of metric slow entropy to arbitrary  probability measure-preserving $\mathbb{R}^k$-actions of a standard probability space and also the vanishing criterion for topological entropy in \cite{KanigowskiVinhageWei} to any free $\mathbb{R}^k$-action that acts by uniformly continuous minimal homeomorphisms of a metric spaces.

Before we state our propositions, let's generalize the slow entropy of an $\mathbb{R}^k$-action with respect to any nondecreasing function $f:\mathbb{R}^+\to\mathbb{R}^+$ with $\lim_{x\to+\infty}f(x)=+\infty$ by replacing $\log$ in Definition \ref{def:topologicalSlow}'s denominator and Definition \ref{def:metricSlow}'s denominator by $f$ and denote them as $h_{f,F_n}^{\operatorname{top}}(\alpha)$ and $h_{f,F_n}^{\mu}(\alpha)$ respectively. Then we can state vanishing propositions for abelian actions as follows:
\begin{proposition}\label{prop:metricVanishing}
Let $(X,\mu)$ be a standard probability space, $F_n$ be a norm-induced F{\o}lner sequence and $\alpha:\mathbb{R}^k\curvearrowright(X,\mu)$ be an action of $\mathbb{R}^k$ by measure-preserving transformations. Then $h_{f,F_n}^{\mu}(\alpha)=0$ for any nondecreasing function $f:\mathbb{R}^+\to\mathbb{R}^+$ with $\lim_{x\to+\infty}f(x)=+\infty$ if and only if $\alpha$ is measurably conjugate to an $\mathbb{R}^k$ action by translations on a compact abelian group.
\end{proposition}
\begin{proof}
The proof of this proposition essentially follows from Proposition $3$ in \cite{Ferenczi}. Notice that if the abelian action $\alpha$ has zero metric slow entropy at all scales, then there is a family $\{\alpha_1,\ldots,\alpha_k\}$ of generators of $\alpha$ with each $\alpha_i$ also having zero slow entropy at all scales. Hence, \cite[Proposition 3]{Ferenczi} implies that all of them are conjugate to some translations on a compact abelian group and thus we obtain the proof for one direction. The other direction directly follows from the fact that translations provide no orbit growth and thus we finish the proof of the proposition.
\end{proof}

\begin{proposition}\label{prop:topologicalVanishing}
Let $(X,d)$ be a compact metric space, $F_n$ be a norm-induced F{\o}lner sequence and $\alpha:\mathbb{R}^k\curvearrowright(X,d)$ be an action by uniformly continuous homeomorphisms of $(X,d)$.  Then $h_{f,F_n}^{\operatorname{top}}(\alpha)=0$ for any nondecreasing function $f:\mathbb{R}^+\to\mathbb{R}^+$ with $\lim_{x\to+\infty}f(x)=+\infty$ if and only if $\alpha$ is topologically conjugate to an $\mathbb{R}^k$ action by translations on a compact abelian group.
\end{proposition}
\begin{proof}
The proof of this proposition is almost identical with the previous one if we replace Proposition $3$ in \cite{Ferenczi} by Proposition A.2 in \cite{KanigowskiVinhageWei}.
\end{proof}

\end{document}